 \documentclass[11pt,reqno]{amsart}  
\usepackage{latexsym}
\usepackage{soul}
\usepackage{amssymb}
\usepackage{mathrsfs}
\usepackage{amsmath}
\usepackage{fancybox,color}
\usepackage{enumerate} 
\usepackage[latin1]{inputenc}
\usepackage{soul} 

\usepackage{amscd,amsthm,amsxtra, array}
\usepackage[dvipsnames]{xcolor}

\usepackage[
colorlinks=true, 
linkcolor=magenta, 
citecolor=magenta,
pagebackref,
]{hyperref}

\usepackage{comment}

\numberwithin{equation}{section}

\usepackage{color}

\usepackage[left=1.8cm,top=2.4cm,right=2.2cm]{geometry}

\def\1{\raisebox{2pt}{\rm{$\chi$}}}

\newtheorem{theorem}{Theorem}[section]
\newtheorem{corollary}[theorem]{Corollary}
\newtheorem{lemma}[theorem]{Lemma}

\newtheorem{definition}[theorem]{Definition}
\newtheorem{remark}[theorem]{Remark}

\newtheorem{lettertheorem}{Theorem}

  \theoremstyle{plain}
  
  \theoremstyle{plain}

\newcommand{\R}{{\mathbb R}}

\newcommand{\N}{{\mathbb N}}

\newcommand{\F}{{\mathcal F}}

\newcommand{\cQ }{\mathcal Q}

\newcommand{\strt}[1]{\rule{0pt}{#1}} %
\newcommand{\car}[1]{\chi_{\strt{1.5ex}#1}}  %

\def\1{\raisebox{2pt}{\rm{$\chi$}}}

\newcommand{\norm}[1]{\left|\left|#1\right|\right|}

\def\vint_#1{\mathchoice%
        {\mathop{\kern 0.2em\vrule width 0.6em height 0.69678ex depth -0.58065ex
                \kern -0.8em \intop}\nolimits_{\kern -0.4em#1}}%
        {\mathop{\kern 0.1em\vrule width 0.5em height 0.69678ex depth -0.60387ex
                \kern -0.6em \intop}\nolimits_{#1}}%
        {\mathop{\kern 0.1em\vrule width 0.5em height 0.69678ex
            depth -0.60387ex
                \kern -0.6em \intop}\nolimits_{#1}}%
        {\mathop{\kern 0.1em\vrule width 0.5em height 0.69678ex depth -0.60387ex
                \kern -0.6em \intop}\nolimits_{#1}}}
\def\vintslides_#1{\mathchoice%
        {\mathop{\kern 0.1em\vrule width 0.5em height 0.697ex depth -0.581ex
                \kern -0.6em \intop}\nolimits_{\kern -0.4em#1}}%
        {\mathop{\kern 0.1em\vrule width 0.3em height 0.697ex depth -0.604ex
                \kern -0.4em \intop}\nolimits_{#1}}%
        {\mathop{\kern 0.1em\vrule width 0.3em height 0.697ex depth -0.604ex
                \kern -0.4em \intop}\nolimits_{#1}}%
        {\mathop{\kern 0.1em\vrule width 0.3em height 0.697ex depth -0.604ex
                \kern -0.4em \intop}\nolimits_{#1}}}

\newcommand{\aveint}[2]{\mathchoice%
        {\mathop{\kern 0.2em\vrule width 0.6em height 0.69678ex depth -0.58065ex
                \kern -0.8em \intop}\nolimits_{\kern -0.45em#1}^{#2}}%
        {\mathop{\kern 0.1em\vrule width 0.5em height 0.69678ex depth -0.60387ex
                \kern -0.6em \intop}\nolimits_{#1}^{#2}}%
        {\mathop{\kern 0.1em\vrule width 0.5em height 0.69678ex depth -0.60387ex
                \kern -0.6em \intop}\nolimits_{#1}^{#2}}%
        {\mathop{\kern 0.1em\vrule width 0.5em height 0.69678ex depth -0.60387ex
                \kern -0.6em \intop}\nolimits_{#1}^{#2}}}

\newcommand{\loc}{\mathrm{loc}}

\newcommand{\vertii}[1]{{\left\vert\kern-0.25ex\left\vert\kern-0.25ex  #1 
    \kern-0.25ex\right\vert\kern-0.25ex\right\vert}}

\newcommand{\vertiii}[1]{{\left\vert\kern-0.25ex\left\vert\kern-0.25ex\left\vert #1 
    \right\vert\kern-0.25ex\right\vert\kern-0.25ex\right\vert}}

\begin{document}

\title{On the BBM-phenomenon in fractional Poincar\'e--Sobolev inequalities with weights}

\author[R. Hurri-Syrj\"anen]{Ritva Hurri-Syrj\"anen}
\address[Ritva Hurri-Syrj\"anen]{Department of Mathematics and Statistics, Pietari Kalmin katu 5, FI-00014 University of Helsinki, Finland}
 \email{ritva.hurri-syrjanen@helsinki.fi}

\author[J. C. Mart\'inez-Perales]{Javier C. Mart\'inez-Perales}
\address[Javier C. Mart\'inez-Perales]{Calle Nueva, 18, Manilva, M\'alaga, Spain} \email{javicemarpe@gmail.com}

\author[C. P\'erez]{Carlos P\'erez}
\address[Carlos P\'erez]{Department of Mathematics, University of the Basque Country, IKERBASQUE 
(Basque Foundation for Science) and
BCAM \textendash  Basque Center for Applied Mathematics, Bilbao, Spain}
\email{cperez@bcamath.org}

\author[A. V.\! V\"ah\"akangas]{Antti V. V\"ah\"akangas}
\address[Antti V. V\"ah\"akangas]{University of Jyvaskyla, Department of Mathematics and Statistics, P.O. Box 35, FI-40014 University of Jyvaskyla, Finland} 
\email{antti.vahakangas@iki.fi}

\thanks{J. M. is supported by the Basque Government through the BERC 2018-2021 program and by the Spanish State Research Agency through BCAM Severo Ochoa excellence accreditation SEV-2017-2018.}

\thanks{C. P. is  supported by the Basque Government through the IT1247-19 project, the BERC 2018-
2021 program  and by the Spanish State Research Agency through BCAM Severo
Ochoa excellence accreditation SEV-2017-0718 and thru the  project 
PID2020-113156GB-I00/AEI /10.13039/501100011033 funded by Spanish State Research Agency and acronym ``HAPDE".}

\thanks{   The authors acknowledge the support
to this project
from the Academy of Finland project  (314829)
of Tuomas Hyt\"onen.  }

\makeatletter
\@namedef{subjclassname@2020}{{\mdseries 2020} Mathematics Subject Classification}
\makeatother

\keywords{Poincar\'e--Sobolev inequality, Muckenhoupt weight.}

\subjclass[2020]{Primary: 46E35. Secondary: 42B25.}

\begin{abstract}

In this paper we unify and improve some of the results of Bourgain, Brezis and Mironescu
 and the   weighted Poincar\'e--Sobolev  estimate by Fabes, Kenig and Serapioni. 
More precisely, we get weighted counterparts of the Poincar\'e--Sobolev type inequality 
 and also of the Hardy type inequality in the fractional case
 under some mild natural restrictions. 

A main feature of the results we obtain is the fact that we keep track of the behaviour of the constants involved when the fractional parameter approaches to $1$. %
Our main method is based on techniques coming from harmonic analysis 
related to the   self-improving property of generalized Poincar\'e inequalities. 

\end{abstract}
\maketitle

\section{Introduction: background and motivation}

Let $p\geq 1$. The classical $(1,p)$-Poincar\'e inequality states the existence of a dimensional constant $c_n>0$ such that, for any Sobolev function $u\in W^{1,p}_\loc(\mathbb{R}^n)$,
\begin{equation}\label{e.11p}
\vint_Q \lvert u(x)-u_Q\rvert \,d x  \le  c_n\, \ell(Q)\left(\vint_Q |\nabla u(x)|^p\, d x\right)^{1/p}, \qquad Q\in \cQ\,,
\end{equation}
where $\cQ$ is the  of  family cubes (i.e. a cartesian product of $n$ intervals of the same side length $\ell(Q)$  in $\R^n$ and $u_Q:=\vint_Q u(x)\,dx:={\frac{1}{|Q|}}\int_Qu(x) dx$). The  mean oscillations of Sobolev functions $u$ over cubes are then controlled 
by the 
local  Sobolev seminorm 
\[
[u]_{W^{1,p}(Q)}:=\ell(Q)\left(\vint_Q|\nabla u(x)|^p\, dx\right)^{1/p}.
\]
More recently (see for instance \cite{DD, DM1, ACPS1, ACPS2,HL}),  a   fractional  counterpart  of this Poincar\'e inequality  has attracted the attention of many authors. A naive version of it states  the existence of a dimensional positive constant $c_n$ such 
that, given $u\in L^1_{\loc}(\mathbb{R}^n)$,
\begin{equation}\label{FPI-rough}
\vint_Q \lvert u(x)-u_Q\rvert\,dx \leq  c_n \ell(Q)^\delta\left(\vint_Q\int_Q\frac{|u(x)-u(y)|^p}{|x-y|^{n+\delta p}}\, dy\, dx\right)^{1/p}, 
\qquad Q\in \cQ 
\end{equation}
for any $p\geq 1$ and $\delta\in(0,1)$. This inequality allows to control the  mean  {oscillation of a function $u$ over a cube $Q$} 
by %
the fractional Sobolev seminorm 
\begin{equation}\label{fractSobolevseminorm}
[u]_{W^{\delta,p}(Q)}:= \ell(Q)^\delta \left(\vint_Q\int_Q \frac{|u(x)-u(y)|^p}{|x-y|^{n+\delta p}}\, dy\, dx\right)^{1/p}.
\end{equation}

Inequality \eqref{FPI-rough} is an easy-to-get estimate but it turns out that it encodes a lot of information. Indeed, it can be shown, by using the methods first proved in \cite{FPW1} and then  improved in  \cite{CP},  that the left hand side can be replaced by 
the normalized weak or Marcinkiewicz norm (see \eqref{normMarcnorm}). We then get 
\begin{equation} \label{weaktypeSobolev}
\big\|u-u_Q\big\|_{L^{p^*_{\delta},\infty}\big( Q, \frac{dx}{|Q|}\big)} \leq c_n \, p^*_{\delta}\, \, [u]_{W^{\delta,p}(Q)}, \qquad Q\in \cQ,
\end{equation}
 whenever  $\delta\in(0,1)$, $p\in [1,\tfrac{n}{\delta})$,   and where $p_{\delta}^*$ is the fractional Sobolev exponent defined by
\begin{equation} \label{FracSobExp}
\frac{1}{p} -\frac{1}{ p_{\delta}^* }=\frac{\delta}{n}.
\end{equation}
Observe that \,$p<\frac{n}{\delta}$ and $p<p_{\delta}^*$. Inequality \eqref{weaktypeSobolev}   readily follows from the following property of the functional $a: \cQ\to [0,\infty)$, $a(Q)=[u]_{W^{\delta,p}(Q)}$,   
\begin{equation}\label{Dp}
\left( \sum_i a(Q_i)^{p^*_{\delta}} \frac{|Q_i|}{|Q|}  \right)^{\frac1{p^*_{\delta}}} \leq a(Q)    
\end{equation}
which holds for any $Q\in \cQ$ and any family of disjoint dyadic subcubes $\left\lbrace Q_i\right\rbrace \in \mathcal D(Q)$. We will
denote the collection of all dyadic cubes by $\mathcal D$ and by $\mathcal D(Q)$ the collection of all dyadic cubes relative to the cube $Q$.

 Moreover,  since the truncation argument works for the functional $[u]_{W^{\delta,p}(Q)}$ as shown in \cite{DIV}, then we can replace the weak norm in 
\eqref{weaktypeSobolev} by the strong norm  to get  
\begin{equation}\label{FracstrongtypeSobolev}
\inf_{c\in \mathbb{R}}\big\|u-c\big\|_{L^{p^*_{\delta}}\big( Q, \frac{dx}{|Q|}\big)} \leq c_n \, p^*_{\delta}\, \, [u]_{W^{\delta,p}(Q)}. 
\end{equation}

It turns out that \eqref{FracstrongtypeSobolev}  is far from being optimal. 
Indeed, it follows from \cite{BBM3},
\begin{equation} \label{FPI with gain}
\vint_Q |u-u_Q|\,dx\leq c_{n}\,%
(1-\delta)^{\frac1p}\, \,[u]_{W^{\delta,p}(Q)},
\end{equation}
where the highly  interesting extra gain 
 $(1-\delta)^{\frac1p}$ appears in front of $[u]_{W^{\delta,p}(Q)}$. A different and interesting approach was considered later in \cite{Mi} combining ideas from  interpolation theory and extrapolation theory \cite{JM}, \cite{DM2}.  
  We remit to \cite{KMX} for interesting extensions to the context of higher-order Besov norms. See also \cite{DM3} for more about the central role played by interpolation theory and extrapolation theory.     
See also \cite{CMPR} for some related  results within the context of product spaces.   

The estimate \eqref{FPI with gain} will   be the    ``key initial" starting point in most of our proofs.

Note that, according to \cite[Proposition 2]{B},  a measurable non-constant function  $u$ on a cube $Q$ would satisfy, 
\begin{equation*}\label{alternative_seminorm}
\left(\vint_Q\int_Q \frac{|u(x)-u(y)|^p}{|x-y|^{n+p}}\, dy\, dx\right)^{\frac1p}=\infty,
\end{equation*}
and so inequality \eqref{FPI-rough} does not provide any information about the function $u$ when $\delta \to 1$. This is corrected in estimate  \eqref{FPI with gain} or \eqref{e.Bgain} below, where the factor $(1-\delta)^{1/p}$ balances this behaviour when $\delta\to 1$ and  so its presence in the inequality is in fact essential.   

Now, exactly as outlined above, to prove \eqref{FracstrongtypeSobolev}, the results in \cite{CP} combined with the truncation method in this context obtained in \cite{DIV}, yield the following result.

\begin{lettertheorem} \label{FracSobGain} Let $0<\delta<1$, $1\leq p<\tfrac{n}{\delta}$ and let $p^*_{\delta}$ be the fractional Sobolev exponent  \eqref{FracSobExp}. Then, for any locally integrable function $u$,
\begin{equation}\label{e.Bgain}
\left(\vint_Q |u(x)-u_Q|^{p^*_{\delta}}\,dx\right)^{\frac{1}{p^*_{\delta}} }\leq c_n \, p^*_{\delta}\, %
(1-\delta)^{\frac1p}\, [u]_{W^{\delta,p}(Q)}   \qquad Q \in \cQ.
\end{equation}
\end{lettertheorem}

One of the main results of this paper is to extend \eqref{e.Bgain} into the context of $A_p$ weights. Unfortunately we can only derive sharp results, preserving the 
$\delta \to 1$, 
in the case of $A_1$ weights.

\begin{remark}\label{globalresult}
Standard arguments can be used to  obtain from  \eqref{e.Bgain}  the global estimate in $\mathbb{R}^{n}$ with the correction factor  %
$(1-\delta)^{\frac1p}$ in front, namely
\[
\left( \int_{\mathbb{R}^n} |u(x)|^{p^*_{\delta}}\,dx  \right)^{\frac{1}{p^*_{\delta}}}
\leq c_n\,p^*_{\delta}\, %
(1-\delta)^{1/p} \,\left(\int_{\mathbb{R}^n}\int_{\mathbb{R}^n}\frac{|u(x)-u(y)|^p}{|x-y|^{n+\delta p}}\, dy\, dx\right)^{1/p},
\]
which holds for, say, any function 
$u\in L^1_{\loc}(\mathbb{R}^{n})$    such that for an increasing family of cubes $\{Q_j\}$ 
with $\R^n=\bigcup_{j=1}^\infty Q_j$  we have  $u_{Q_j}\to 0$ as $j \to \infty$,  see Definition \ref{defmathcalF} in Section \ref{2.1} below. 
Corresponding inequalities are considered for functions defined 
on unbounded John domains in \cite{HV2}
without keeping track of the constants' exact dependence on the parameter $\delta$.
Different and interesting approaches to prove this global result were provided in \cite{MS1}  and \cite{KL}.
\end{remark}

A global version of Theorem  \ref{FracSobGain} was obtained in \cite{MS1} using appropriate global Hardy type estimates  which are also interesting on its own and that we state next.

\begin{lettertheorem}[\text{\cite[Theorem 2]{MS1}}]\label{t.MS.Hardy} Let $1\leq p<\infty$ and $0<\delta<1$ such that $\delta p<n$. There exists $c_{n,p}>0$ such that, for any function in the completion $W_0^{\delta,p}(\mathbb{R}^n)$   
of the space of compactly supported smooth functions under the seminorm $[\ \cdot\ ]_{W^{\delta,p}(\mathbb{R}^n)}$,  there holds
\begin{equation}\label{MS.Hardy}
\left(\int_{\mathbb{R}^n}|u(x)|^p\frac{dx}{|x|^{\delta p}}\right)^{1/p}\leq 
c_{n,p} \, \, \delta^{1/p}(1-\delta)^{1/p} \left(\int_{\mathbb{R}^n}\int_{\mathbb{R}^n}\frac{|u(x)-u(y)|^p}{|x-y|^{n+\delta p}}\, dy\, dx\right)^{1/p}.
\end{equation}
\end{lettertheorem}
We improve this inequality in this work. First we will replace   the class of power weights 
$\frac{1}{|x|^{\delta p}}$ at the left hand side of \eqref{MS.Hardy} by   $A_{\infty}$ weights and then with no assumption on the weight  with a worse constant in front. Our method is different and more general.

 In the following section we state and discuss the main results of this paper. We distinguish between the weighted variants of \eqref{e.Bgain} and \eqref{MS.Hardy}, which we will call fractional  Poincar\'e--Sobolev and Hardy type inequalities, respectively.
 
 \

{\bf Acknowledgements} 

Part of the research leading to the results in this article took place at the  Department of Mathematics of Aalto University while the third author was a visiting faculty. He is very grateful to the Prof. Juha Kinnunen for his invitation,  his hospitality and for the many  discussions we had. He is  also very grateful to his PhD students Julian Weigt and Kim Myyrylainen for their interest and for many interesting discussions. In particular, we are very grateful to Julian Weigt for pointing out to us an incorrect  stronger variant of estimate \eqref{FPI with gain} which appeared  
in the first version of this paper with the extra factor $\delta^{\frac{1}{p}}$ in the right-hand side for any $\delta\in (0,1)$.

\section{ Statement of the main results }  \label{subsec.weighted.PS}

\subsection{The case of $A_{\infty}$ type weights} \label{2.1} 

Our first main result is an extension of Theorem \ref{FracSobGain} to include weights from the $A_1$ class.  
  The method of proof  in \cite{BBM3} cannot be used at all due to the presence of the weight. We use ideas from \cite{PR} instead,  where a  general ``self-improving" argument is introduced, thus avoiding  completely the use of any representation formula.
Actually, we will state two type of Poincar\'e--Sobolev inequalities. The first one involves a {\it weighted fractional Sobolev exponent}, $p_{\delta,w}^*$ and the second  one involves the   usual    fractional Sobolev exponent $p_\delta^*$ \eqref{FracSobExp}.

\begin{theorem} \label{selfIMproveGoodConstant}  
Let $0<\delta<1$ and $1\leq p< \tfrac{n}{\delta}$. 
 Let $w \in A_1 $,  and let 
\,$p_{\delta,w}^*$ be the weighted fractional Sobolev exponent defined by
\begin{equation}\label{DegSobExp}
\frac{1}{p} -\frac{1}{ p_{\delta,w}^* }=   \frac{\delta}{n} \, \frac{1}{1+\log [w]_{A_1}}.
\end{equation}
Then there is a dimensional constant $c_n$ such that, for every cube $Q$ in $\mathbb{R}^n$  and for any $u\in L^1_\loc(\mathbb{R}^n)$, 
\begin{equation}\label{A1PSFractBBM}
\begin{split}
\inf_{c\in \mathbb{R}} & \left( \frac{1}{w(Q)}  \int_{ Q }   |u -c|^{p_{\delta,w}^*}     \,wdx\right)^{\frac{1}{p_{\delta,w}^*}}  
\\&\qquad\leq c_n\,  p_\delta^*\,  %
(1-\delta)^{\frac1p} \,  [w]^{1+\frac{1}{p}}_{A_1}\,  \ell(Q)^{\delta} \,
\left(   \frac{1}{w(Q)} \int_{Q}  \int_{Q}  \frac{|u(x)- u(y)|^p}{|x-y|^{n+p\delta}}\,dy\,wdx\right)^{\frac1p}.  
\end{split}
\end{equation}
\end{theorem}

\begin{remark} %
We emphasize that in spite of the singularity introduced by the weight $w$, the smallness factor 
$(1-\delta)^{\frac1p}$ is kept as if there were no such singularity.
\end{remark}

This theorem and next one can be seen as improvement of the celebrated Poincar\'e--Sobolev inequalities obtained by Fabes, Kenig and Serapioni \cite{FKS} (see also \cite{HKM})  by combining this result and the ``reverse type'' results obtained in \cite{HMPV}, at least in the case $w\in A_1$. 

Observe that  \,$p<p_{\delta,w}^* \leq p_{\delta}^*$ where $p^*_{\delta}$ is the fractional  Sobolev exponent defined above and note that $p_{\delta,w}^*$ is of the form \[ p_{\delta,w}^*:=\frac{pn(1+\log[w]_{A_1})}{n(1+\log[w]_{A_1})-\delta p},\]
and so, by comparing with the expression of   $p_\delta^*=\frac{np}{n-p\delta}$ , the term $1+\log[w]_{A_1}$ may be regarded as a distortion of the dimension $n$ introduced by the presence of the weight $w$. Nevertheless, the largest possible borderline exponent $p_\delta^*$ can also be attained with the presence of an $A_1$ weight.  
The cost of this better  improvement  in the scale of the $L^p$ spaces is the loss of some extra $A_1$ constant in front.

\begin{theorem}\label{selfIMproveBadConstant} 
 Let $0<\delta<1$ and $1\leq p< \tfrac{n}{\delta}$.  Let $w \in A_1 $,  and let
$p^*_{\delta}$ be the fractional Sobolev exponent \eqref{weaktypeSobolev}.
Then,  there exists a constant $c_{n}>0$ such that for every $Q\in \cQ$  and for any $u\in L^1_\loc(\mathbb{R}^n)$,
\begin{align*}
&\inf_{c\in \mathbb{R}}\left (\frac{1}{w(Q)} \int_Q |u-c|^{p_{\delta}^*} w dx\right )^{\frac{1}{p_{\delta}^*}} 
\\&\qquad \leq 
c_n\,p_{\delta}^*\,\, %
(1-\delta)^{\frac1p} \,
 [w]_{A_1}^{   \frac{\delta}{n} +1+\frac1p}  
 \ell(Q)^{\delta} \, 
\left(\frac{1}{w(Q)} \int_{Q}  \int_{Q}  \frac{|u(x)- u(y)|^p}{|x-y|^{n+p\delta}}\,dy\,wdx\right)^{\frac1p}.
\end{align*}

\end{theorem}

These two theorems hold for any $A_p$ weight without the 
gain  %
$(1-\delta)^{\frac1p}$. It may well be the case that both results with $A_p$ weights are true with the %
gain.

Another  satisfactory result is the following fractional Hardy type inequality which follows, as mentioned before, from a general self-improving argument which avoids completely the use of any representation formula.

\begin{theorem} \label{t.hardy-MS}     Let $0<\delta< 1$ and let  $1\leq p<\infty$.   
There exists a dimensional constant $c_n>0$  such that,  for any $w\in A_\infty$  and for any $u\in L^1_\loc(\mathbb{R}^n)$, 
\begin{equation}\label{BBMHardyIneq-Fractp>1}
 \inf_{c\in\R}\left(\int_{Q} \lvert u(x)-c\rvert^p\,w(x)dx\right)^{\frac1p}
\leq c_{n}\,p\,   %
(1-\delta)^{\frac1p}\, [w]_{A_\infty}\,
\left(\int_{Q} \int_{Q} \frac{\lvert u(x)-u(y)\rvert^p}{\lvert x-y\rvert^{n+\delta p}}\,dy\, M_{\delta p,  {Q}}(w)(x)dx\right)^{\frac{1}{p}}
\end{equation}
for every cube $Q$ in $\mathbb{R}^n$. 
\end{theorem}

Observe that this result is of different nature since the fractional part has been absorbed by the fractional maximal function $M_{\delta p,  {Q}}$ (see Definition \ref{fracmax}). 
This is the reason why we cannot get an $``L^{p_\delta^*}"$, namely a Poincar\'e--Sobolev, version as in the other theorems.  However, we can derive a global type result for which the following family of class of functions is relevant.

\begin{definition}\label{defmathcalF}   
Let $w$ be a weight in $\mathbb{R}^n$. We define $\F_{w}$ as the class of functions $u$ such that $u\in L^1_{\loc}(\R^n) \cap L^1_{\loc}(\R^n;wdx)$ for which there exists an increasing sequence of cubes $\{Q_j\}_{j\in \mathbb{N}}$ such that 
\[
 \mathbb{R}^n=\bigcup_{j\in\mathbb{N}} Q_j \quad \text{and} \quad \lim_{j\to\infty} u_{\strt{1.5ex}Q_j,w}=0,
\]
where\, $u_{\strt{1.5ex}Q,w}=\frac{1}{w(Q)}\int_Q u\,wdx.$
\end{definition}

Notice that $L_c^{\infty}(\R^n) \subset \F_{w}$ whenever $w$   is a weight outside of $L^1(\mathbb{R}^n)$ .
 Doubling weights, that is, weights
satisfying $0<w(B(x,2r))\le C(w)w(B(x,r))$ for
all $x\in \R^n$ and $r>0$, are typical examples
of weights that do not belong to $L^1(\R^n)$.
We refer to \cite[Corollary 3.8]{BB}.

\begin{corollary}\label{globalversion1}  
Let $0<\delta<1$ and let  $1\leq  p<\infty$. 
There exists a dimensional constant $c_n>0$ such that, for any $w\in A_\infty$,
 \begin{equation}
\left(\int_{\mathbb{R}^n} \lvert u(x) \rvert^p wdx\right)^{1/p}\leq  c_n\, p\,%
(1-\delta)^{1/p}
[w]_{A_{\infty}}\, \left(\int_{\mathbb{R}^n} \int_{\mathbb{R}^n} \frac{\lvert u(x)-u(y)\rvert^p}{\lvert x-y\rvert^{n+\delta p}}\,dy\, M_{\delta p}(w)(x)dx\right)^{1/p}
\end{equation}
 whenever $u\in \F_w$. 
\end{corollary}

The $A_\infty$ restriction is sufficiently mild to allow  many   interesting examples. In particular, $A_1$ weights can be chosen for our result  including power weights.  The fact that the    result in Theorem \ref{t.hardy-MS}    does not involve any geometric constant depending on $\ell(Q)$   is what    allows to  get  a global variant of it,    thus getting   
Theorem \ref{t.MS.Hardy} as a corollary of ours at least in the case $\delta\to1$.

More precisely, as a  consequence of  Theorem \ref{t.hardy-MS}  we get the following general local Hardy type inequality.

\begin{corollary}\label{ThmSpecialcase}
Let $0<\delta<1$ and $\beta>0$, and let  $1\leq p<\frac{n}{\delta}\min\{1,\beta\}$. Let  $\mu$ be any non-negative  Borel  measure such that\, $M_{n-\frac{p\delta}{\beta}}(\mu)$ is finite almost everywhere. Then, there exists a positive dimensional constant $c_n$  such that, for any  $u\in L^1_{\mathrm{loc}}(\R^n)$,
\begin{equation}\label{SpecialFractPI1}
\begin{split}
 \inf_{c\in\R} & \left(\int_{Q} \lvert u(x)-c\rvert^p \,M_{n-\frac{p\delta}{\beta}}(\mu)(x)^{\beta}\,dx\right)^{\frac{1}{p}}
\\&\qquad \leq c_n^{\beta+1}\, %
(1-\delta)^{\frac{1}{p}} \,\frac{p}{(n-\delta p)^{1+\frac1p} }\, \mu(\mathbb{R}^n)^{\frac{\beta}{p}}\, \left(\int_{Q} \int_{Q} \frac{\lvert u(x)-u(y)\rvert^p}{\lvert x-y\rvert^{n+\delta p}}\,dy\,dx\right)^{\frac{1}{p}}.
\end{split}
\end{equation}
In particular if we choose $\beta=1$ and $\mu$ as the Dirac mass at the origin, we have that 
\begin{equation}\label{SpecialFractPI2}
\inf_{c\in\R}\left(\int_{Q} \lvert u(x)-c\rvert^p \frac{1}{|x|^{\delta p}}\,dx\right)^{\frac{1}{p}}
\leq c_n\,  %
(1-\delta)^{\frac{1}{p}} \,  \frac{p}{(n-\delta p)^{1+\frac1p} }\,    \left(\int_{Q} \int_{Q} \frac{\lvert u(x)-u(y)\rvert^p}{\lvert x-y\rvert^{n+\delta p}}\,dy\,dx\right)^\frac{1}{p}.
\end{equation}
\end{corollary}

We notice that $M_{n-\frac{p\delta}{\beta}}(\mu)^{\beta}$ is a weight that belongs to the $A_1$ class since $p<\frac{n}{\delta}\min\{1,\beta\}$. Indeed, whenever $0\leq \alpha<n$, $\rho>0$,  it is known that $M_{\alpha}(\mu)^{\rho}\in A_1$ if $0<\rho<\frac{n}{n-\alpha}$ (see Lemma \ref{PropertiesFractMax.} (2) for more precise details).

The global versions follow easily. 
 
 \begin{corollary}\label{ThmSpecialcaseg}
 Let $0<\delta<1$ and $\beta>0$,   and let  $1\leq p<\frac{n}{\delta}\min\{1,\beta\}$. Let $\mu$ be any non-negative Borel measure  such that\, $M_{n-\frac{p\delta}{\beta}}(\mu)$ is finite almost everywhere. Then there exists a dimensional positive constant $c_n$  such that 
\begin{equation}\label{SpecialFractPI1g}
\begin{split}
&\left(\int_{\mathbb{R}^n} \lvert u(x) \rvert^p (M_{n-\frac{p\delta}{\beta}}\mu)(x)^{\beta}\,dx\right)^{\frac{1}{p}}
\\&\qquad \qquad \leq  c_n^{\beta+1}\, %
(1-\delta)^{\frac{1}{p}} \, \frac{p}{(n-\delta p)^{1+\frac1p} } \,   \mu(\mathbb{R}^n)^{\frac{\beta}{p}}\, \left(\int_{\mathbb{R}^n} \int_{\mathbb{R}^n} \frac{\lvert u(x)-u(y)\rvert^p}{\lvert x-y\rvert^{n+\delta p}}\,dy\,dx\right)^{\frac{1}{p}}.
\end{split}
\end{equation}
whenever $u\in \F_w$ where $w=(M_{n-\frac{p\delta}{\beta}}\mu)^{\beta}$.
Hence, if $\beta=1$ and  $\mu$ is the Dirac mass at the origin, we have that 
\begin{equation}\label{SpecialFractPI2g}
\left(\int_{\mathbb{R}^n} \lvert u(x) \rvert^p \frac{1}{|x|^{\delta p}}\,dx\right)^{\frac{1}{p}}
 \leq c_n\, \, %
 (1-\delta)^{\frac{1}{p}} \, 
 \frac{p}{(n-\delta p)^{1+\frac1p} } \,    \left(\int_{\mathbb{R}^n} \int_{\mathbb{R}^n} \frac{\lvert u(x)-u(y)\rvert^p}{\lvert x-y\rvert^{n+\delta p}}\,dy\,dx\right)^{\frac{1}{p}},
\end{equation}
whenever $u\in \F_w$ where $w=\frac{1}{|x|^{\delta p}}$.
\end{corollary}

We conjecture that  the correct bound 
in terms of $\delta\in (0,1)$  is   actually    $\delta^{\frac{1}{p}} \,(1-\delta)^{\frac1p}$.
We remit to Section \ref{self-improvingmethods} for the proofs of the results of this section.

\subsection{ Weighted   fractional Poincar\'e    type inequalities without the $A_{\infty}$ condition} \label{subsec.weighted.Hardy}

In this section, we state  several extensions of Theorem \ref{t.MS.Hardy}  where power weights are replaced by a general weight $w$. In particular we do not assume that the weight satisfies the $A_{\infty}$  condition. However,  we cannot obtain the  factor $\delta^{1/p}$ in \eqref{MS.Hardy} from Theorem \ref{t.MS.Hardy}.

The results of this section are motivated by the  classical 
Fefferman-Stein  inequality \cite{FS}, 
\begin{equation*}%
\|Mf\|_{\strt{2ex}L^{1,\infty}(w)} \le c_n \,
 \int_{\R^{n}} |f(x)|\,Mw(x)dx.
\end{equation*}
from which we  deduce, for $p\in (1,\infty)$, that 
\begin{equation*}%
\|Mf\|_{\strt{2ex}L^{p}(w)}    \leq c_n\,p'\, \|f\|_{\strt{2ex}L^{p}(Mw)},
\end{equation*}
where no assumption is made on the weight $w$.

\begin{theorem} \label{A1-two weight}     
Let $1\leq p<\infty$ and let $w$ be any weight. %

a) Let $0<\delta< 1$  and let $1\leq p<\infty$. Let also $0<\varepsilon \leq  \delta$, then there is a positive dimensional constant $c_n$ such that for any weight $w$,

\begin{equation}\label{A1-two weight1} 
\begin{aligned}
\Bigg{(}\int_{Q} &\lvert u(x)-u_{Q}\rvert^p \,wdx\Bigg{)}^{1/p}
\\
&\leq c_n\, \frac{ %
(1-\delta)^{\frac{1}{p}}}{ \varepsilon } \ell(Q)^{\varepsilon }
\left(\int_{Q} \int_{Q} \frac{\lvert u(x)-u(y)\rvert^p}{\lvert x-y\rvert^{n+\delta p}}\,dy\, M_{(\delta-\varepsilon)p,  {Q}}(w {\chi_Q})(x)\,dx\right)^{1/p}.
\end{aligned}
\end{equation}
b) Hence, if $\frac12<\delta< 1$,   
\begin{equation}\label{FS-type} 
\begin{aligned}
\Bigg{(}\int_{Q} &\lvert u(x)-u_{Q}\rvert^p \,wdx\Bigg{)}^{1/p}
\\
&\leq c_n\, (1-\delta)^{\frac{1}{p}}  \ell(Q)^{\delta }
\left(\int_{Q} \int_{Q} \frac{\lvert u(x)-u(y)\rvert^p}{\lvert x-y\rvert^{n+\delta p}}\,dy\, M(w {\chi_Q})(x)\,dx\right)^{1/p}
\end{aligned}
\end{equation}
and in particular, for any $(w,v)\in A_1$, we have
\begin{equation}\label{A1-two weight2} 
\begin{aligned} \Bigg{(}\int_{Q} &\lvert u(x)-u_{Q}\rvert^p \,wdx\Bigg{)}^{1/p} \\
&\leq c_{n}\, (1-\delta)^{\frac{1}{p}} \,\ell(Q)^{\delta } [w,v]_{A_1}^{\frac1p}
\left(\int_{Q} \int_{Q} \frac{\lvert u(x)-u(y)\rvert^p}{\lvert x-y\rvert^{n+\delta p}}\,dy\,  v\,dx\right)^{1/p}.
\end{aligned}
\end{equation}

\end{theorem}

\subsection{Weighted fractional  isoperimetric inequalities  with one sharp gain} \label{subsec.weighted.Isop.Ineq} 
 
We also {consider}    local and global \textit{fractional} versions of the  following weighted Gagliardo-Nirenberg inequality 
\begin{equation}\label{isoperimetric}
\left(\int_{\mathbb{R}^n} |u(x)|^{\frac{n}{n-1}}w(x)dx\right)^{\frac{n-1}{n}}\leq c_n\,\int_{\mathbb{R}^n} |\nabla u(x)|(Mw(x))^{\frac{n-1}{n}}dx,
\end{equation}
which is a result that already appeared in \cite{PR} and was also implicit in \cite{FPW2}.  This generalizes the well known Gaglardo-Nirenberg-Sobolev inequality in the limiting case $p=1$, which in turn is equivalent to the  well known isoperimetric inequality, see \cite{O,Fed,M}.   

We obtain a local fractional version of \eqref{isoperimetric},  with the presence of the  sharp gain phenomenon when $\delta \to 1$. 
Observe that   the admissible weights here are more singular since there is no assumption on the weight $w$. 

\begin{theorem} \label{strong_local_isoperimetric}
Let $w$ be a weight in $\R^n$ 
and let $\frac12\leq  \delta<1$. Then there exists a dimensional constant $c_n$ such that 
\begin{equation}\label{weightedLocalIsoper}
 \inf_{c\in\R} \left (\int_Q |u(x)-c|^{\frac{n}{n-\delta}}w(x)\, dx\right )^\frac{n-\delta}{n} \leq c_n\,(1-\delta)  
 \int_Q \int_{Q} \frac{\lvert u(x)-u(y)\rvert}{\lvert x-y\rvert^{n+\delta}}\,dy\, Mw(x)^\frac{n-\delta}{n}\,dx, 
\end{equation}
for any cube $Q$ and any $u\in L^1_{\mathrm{loc}}(\mathbb{R}^n)$. 
\end{theorem}

The absence of any quantity depending on the size of the cube yields the following global results.
 
\begin{corollary}\label{globalversion2} 
Let $w$ be a weight in $\R^n$ 
and let $\frac12\leq  \delta<1$.   There is a positive dimensional constant $c_n$ such that, 
\begin{equation*}
\left (\int_{\mathbb{R}^n}  |u(x)|^{\frac{n}{n-\delta}}w(x)\, dx\right )^\frac{n-\delta}{n}\leq  c_n \,(1-\delta)\, \int_{\mathbb{R}^n}    \int_{\R^n} \frac{\lvert u(x)-u(y)\rvert}{\lvert x-y\rvert^{n+\delta }}\,dy \,(Mw(x))^\frac{n-\delta}{n}\, dx,
\end{equation*}
whenever $u\in \F_w$.

\end{corollary}

\begin{corollary}\label{fractionalCharacteristics}

Let $w$ be a weight in $\R^n$ 
and let $\frac12\leq  \delta<1$. Assume further that $w\notin L^1(\R^n)$.    Then there is a positive dimensional constant $c_n$ such that   
\begin{align*}
w(E)^{\frac{n-\delta}{n}}\leq  c_n\, (1-\delta) \,  \Bigg(&\int_E\int_{\R^n\setminus E} \frac{1}{\lvert x-y\rvert^{n+\delta}}\,dy\,(Mw(x))^\frac{n-\delta}{n}\,dx\\
&\qquad  +\int_{\R^n\setminus E}\int_{E} \frac{1}{\lvert x-y\rvert^{n+\delta}}\,dy\,(Mw(x))^\frac{n-\delta}{n}\,dx
\Bigg)
\end{align*}
for every bounded measurable set $E\subset \R^n$.
\end{corollary}

To finish the section we provide another consequence of  Theorem \ref{strong_local_isoperimetric}. 

\begin{corollary}\label{BBMwithweights} Let $Q$ be a cube and let $E$ be any measurable set $E\subset Q$. Then, for any weight $w$  and any $0<\varepsilon\leq\frac12$, 
\[
w(Q\setminus E)\,w(E) \leq w(Q)\,\left(
c_n\varepsilon \int_E \int_{Q\setminus E} \frac{Mw(x)^\frac{n-1+\varepsilon}{n}}{\lvert x-y\rvert^{n+1-\varepsilon}}\,dydx
+
c_n\varepsilon \int_{Q\setminus E} \int_E  \frac{Mw(x)^\frac{n-1+\varepsilon}{n}}{\lvert x-y\rvert^{n+1-\varepsilon}}\,dydx
\right) ^{\frac{n}{n-1+\varepsilon}}.
\]
\end{corollary}

This result is  an extension     of  \cite[Corollary 1] {BBM3} with   $w=1$ and $Q$ the unit cube. This, in turn, is related to an improvement of an estimate from \cite{BBM2}. The improved version was used to prove a conjecture formulated in \cite{BBM1}.  Our result should play a similar role.

\section{Preliminaries and known results} \label{preliminApWeights}

\subsection{Some basic facts about the $A_p$ theory of weights. }

In this section we recall first some well known definitions.

We recall briefly some concepts about the   classes of Muckenhoupt   weights. A weight is a function $w\in L^1_{\mathrm{loc}}(\mathbb{R}^n)$  satisfying $w(x)> 0$ for  almost every point  $x\in\mathbb{R}^n$. 

\begin{definition}  Let $w\in L^1_{\mathrm{loc}}(\mathbb{R}^n)$ be a weight.  

\begin{enumerate}
\item For $1<p<\infty$ we say that $w\in A_p$ if 
\[
[w]_{A_p} := \sup_{Q\in \cQ} \vint_Q w(x)dx\left(\vint_Q w(x)^{1-p'}dx\right)^{p-1}<\infty.
\]
\item We say that $w\in A_1$ if there is a constant $C>0$ such that, for any cube $Q$ in $\mathbb{R}^n$,
\[
\frac{1}{|Q|} \int_Q w(x)dx\leq C\,\mathrm{ess\ inf}_{x\in Q}w(x).
\]
and the $A_1$ constant $[w]_{A_1}$ is defined as the smallest of these $C$.
\item The $A_{\infty}$ class is defined as the union of all   the    $A_p$ classes, that is,
\[ A_\infty = \bigcup_{1\leq p<\infty} A_p,
\]
and the $A_\infty$ constant is defined as
\[ [w]_{A_\infty} = \sup_{Q\in \cQ} \frac{1}{w(Q)} \int_Q M(w\car Q)(x)\,dx.\]
\end{enumerate} 
 \end{definition}

We will use the following notation for the weighted local $L^q$ average over a set $E$ defined as follows:
\begin{equation*}
\left\|f\right \|_{L^q\left (E,\frac{w\, dx}{w(E)}\right )}:= \left (\vint_E |f |^q w dx\right )^{\frac{1}{q}}:=
 \left (\frac{1}{w(E)}\int_E |f |^q w dx\right )^{\frac{1}{q}}.
\end{equation*}

Similarly, we will use the standard notation for the normalized weak $(q,\infty)$ (quasi)norm: for any $0<q<\infty$, measurable $E$ and a positive weight $w$, we define the normalized weak or Marcinkiewicz norm
\begin{equation}\label{normMarcnorm}
\| f \|_{L^{q,\infty}\big(E, \frac{w\,dx}{w(E)}\big)}:= \sup_{t>0}  t \, \left( \frac{1}{w(E)}w(\{x\in E:  |f(x)|>t\}) \right )^{\frac1q},
\end{equation}
where, for a measurable set $E$, we denote its weighted measure by $w(E):=\int_Ew(x)\,dx$.   Similarly we define the global 
weak or Marcinkiewicz norm   
\begin{equation}\label{globMarcnorm}
\|f\|_{L^{q,\infty}(w)}:= \sup_{t>0}  t \, \left( w(\{x\in \R^n:  |f(x)|>t\}) \right )^{\frac1q}.
\end{equation}

We will also use few times the class of pair of $A_p$ weights. 
\begin{definition}\label{TwoWeights}
A pair of weights $(w,v)$ belongs to the $A_p$ class, $1<p<\infty$, if
\begin{equation}\label{eq:Ap-two-weights}
[w,v]_{A_p}:=\sup_{Q\in \cQ} \left(\vint_Q w\,dx\right) \left(\vint_Q v^{1-p'}\,dx\right)^{p-1} <\infty.
\end{equation}
We denote this by $(w,v)\in A_p$. 
Similarly, we say that  $(w,v)\in A_1$ if \[[w,v]_{A_1}= \mathrm{ess\ sup}_{x\in \R^n} \frac{Mw(x)}{v(x)}<\infty,
\]
  where $M$ is the usual Hardy-Littlewood maximal function (see Definition \ref{fracmax} with $\alpha=0$).   
\end{definition}

It is also well known that $(w,v)\in A_p$ if and only if 
\begin{equation}\label{two weight Ap charact}
\frac{1}{|Q|}  \int_Q f\ dx \leq [w,v]_{A_p}^{ \frac{1}{p} }  
\left( \frac{1}{w(Q)}\int_Q f^p v\ dx\right )^{\frac{1}{p}}   \qquad f \geq 0.
\end{equation}

\subsection{Weighted oscillation}

We are often estimating the following weighted $L^q$ oscillation %
\begin{equation}\label{LqOsc}
\left (\frac{1}{w(Q)}\int_Q|u-u_{\strt{1.5ex}Q,w}|^{q}wdx\right )^\frac{1}{q},
\end{equation}
where  $1\le q < \infty$ and $w$ is a weight.
Recall that for a weight $w$, we write $u_{\strt{1.5ex}Q,w}=\frac{1}{w(Q)}\int_Q u\,wdx$.  When estimating 
\eqref{LqOsc} from above, it is possible and often more convenient 
to  replace $u_{\strt{1.5ex}Q,w}$ by any other constant $c$ instead. This is a consequence of inequalities
\begin{equation}\label{LqOscProp}
\inf_{c\in \mathbb{R}}\left (\frac{1}{w(Q)}\int_Q|u-c|^{q}wdx\right )^\frac{1}{q} \leq \left (\frac{1}{w(Q)}\int_Q|u-u_{Q,w}|^{q}wdx\right )^\frac{1}{q}\leq 2\inf_{c\in \mathbb{R}}\left (\frac{1}{w(Q)}\int_Q|u-c|^{q}wdx\right )^\frac{1}{q},
\end{equation}
the latter of which follows by combining triangle and Jensen's inequalities.

\subsection{Review of fractional integrals and related maximal functions}
 
We introduce now   both the fractional and the classical variants of the maximal function.

\begin{definition}\label{fracmax}
Let $\alpha\geq0$ and consider a locally integrable function $u\in L^1_{\mathrm{loc}}(\mathbb{R}^n)$. The fractional maximal function of $u$ is defined by
\[
M_\alpha u(x):=\sup_{Q\ni x}\ell(Q)^\alpha \vint_Q |u(y)|dy,
\]
where the supremum is taken over all cubes $Q$ in $\mathbb{R}^n$ satisfying $x\in Q$. The case $\alpha=0$ corresponds to the non-centered Hardy-Littlewood maximal function and we will denote $M=M_0$. 
 Similarly, for any non-negative Borel measure  $\mu$,  the fractional maximal function $M_{\alpha}\mu$ is defined by 
\[
M_{\alpha}\mu(x):=\sup_{x\in Q: Q\in \cQ}\ell(Q)^\alpha \,\frac{\mu(Q)}{|Q|}.
\]

\end{definition}  

\begin{definition}
Let $\alpha\ge 0$.
If $Q_0\subset \R^n$ is a cube, then the   local   fractional maximal function $M_{\alpha,Q_0} $ is the operator defined by
\[
M_{\alpha,Q_0} u(x)=\sup_{x\in Q}
\ell(Q)^\alpha \vint_{Q} \lvert u(y)\rvert\,d  y,\qquad x\in Q_0,
\]
where the supremum is taken over all cubes $Q \subset Q_0 $ such that $x\in Q$ and $u$ is a measurable function. When   the supremum is taken over all cubes $Q \in \mathcal{D}(Q_0)$ such that $x\in Q$ we get the dyadic local fractional maximal function $M_{\alpha,Q_0}^d$.
\end{definition}
 
Another important operator in the theory, which is related to the fractional maximal function is the fractional integral operator, or  the Riesz potential.
\begin{definition} 
For $0<\alpha<n$,  we define the Riesz potential of a non-negative measurable function $u$ by
\[
I_{\alpha} u(x)=\int_{\R^n} \frac{  u(y) }{\lvert x-y\rvert^{n-\alpha}}\,dy\,,\qquad x\in\R^n,
\]
\end{definition}

We finish this part of the section by proving some elementary properties about weights in relation to the fractional maximal function and the fractional integral operator.  
 
\begin{lemma}\label{PropertiesFractMax.}
Let $0\leq\alpha<n$. 
\begin{enumerate} 
\item Let  $\mu$ be a non-negative Borel measure. Then, 
\begin{equation}\label{weaktypeMalpha}
\|M_{\alpha}\mu\|_{L^{\frac{n}{n-\alpha},\infty}(\mathbb{R}^n) } \leq 5^{n-\alpha} \, \mu(\mathbb{R}^n)\,.
\end{equation}
\item  If $0< \varepsilon\leq\frac{n}{n-\alpha}$, then $(M_{\alpha}\mu)^{\frac{n}{n-\alpha} -\varepsilon} \in A_1
$  for any  non-negative Borel measure $\mu$ such that $M_{\alpha}\mu(x)<\infty$ for almost every $x\in \mathbb{R}^n$. The $A_1$ constant does not depend on $\mu$.  More  precisely,   
\begin{equation}\label{A1estimate}
[(M_{\alpha}\mu)^{\frac{n}{n-\alpha} -\varepsilon}]_{A_1} \leq 
 2^{  \frac{n}{n-\alpha} - \varepsilon} \frac{2n}{n-\alpha}\frac{15^{n-\varepsilon(n-\alpha)}}{\varepsilon}\,.
\end{equation}
\item If $f\in L^1_{\mathrm{loc}}(\mathbb{R}^n)$, then  if further $\alpha>0$,
\[
M_{\alpha} (f)(x)\leq \left(\frac{n}{\alpha}\right)'\|f\|_{L^{\frac{n}{\alpha},\infty}(\mathbb{R}^n)}, \qquad x\in \mathbb{R}^n.
\]
\end{enumerate}
\end{lemma}
\begin{proof}
The proof of part $(1)$ follows from the classical $5$-covering lemma: If $\{Q_i\}_{i=1,\ldots,N}$ is a finite family of cubes in $\R^n$. Then we can extract a subsequence of pairwise disjoint cubes $\{Q_j\}_{j=1,\ldots,M}$ such that
\[
\bigcup_{i=1}^{N} Q_i \subset \bigcup_{j=1}^{M} 5Q_j.
\]
Now if we let  $\Omega_{t} = \{x\in \R^n : M_{\alpha}\mu(x)> t \}$ for a given
$t$ and let $K$ to be any compact subset contained in
$\Omega_t$. Let $x \in K$ then by definition of the maximal function there is a cube $Q=Q_x$ containing $x$ such that
\begin{equation}
\frac{|Q|^{\frac{\alpha}{n}}}{ |Q| } \mu(Q) > t.
\label{ineq}
\end{equation}
Then $K \subset \bigcup_{x \in K } Q_x $ and by compactness we can
extract a finite family of cubes 
$\{Q_i\}_{i=1,\cdots,N}$ such that  $K
\subset \bigcup_{i=1,\cdots,N} Q_i$ and where each cube satisfies
\eqref{ineq}. Then by the  $5$-covering lemma we can extract a subsequence of pairwise disjoint family of cubes 
$\{Q_j\}_{j=1}^{M}$ such that $ \bigcup_{i=1}^{N} Q_i \subset \bigcup_{j=1}^{M} 5Q_j$. Then, since $|\lambda Q|=\lambda^n |Q|$  and since $0<\frac{n-\alpha}{n}\leq 1$
\[
|K|^{\frac{n-\alpha}{n}}\leq \sum_{j=1}^{M} |5Q_j|^{\frac{n-\alpha}{n}}
\leq 5^{n-\alpha}
\sum_{j=1}^{M} |Q_j|^{\frac{n-\alpha}{n}}
\leq
\frac{5^{n-\alpha}}{t}\sum_{j=1}^{M}  
\mu(Q_j)
\leq
\frac{5^{n-\alpha}}{t}  \mu(\R^n).
\]
This yields \eqref{weaktypeMalpha} immediately, since $K$ is arbitrary.

We now prove Part (2). The fact that $(M_{\alpha}\mu)^{\frac{n}{n-\alpha} -\varepsilon}$ is an $A_1$ weight under our assumptions on the parameters is known but we need a proof with precise  bounds as given by \eqref{A1estimate}. Actually we will assume $0< \varepsilon<\frac{n}{n-\alpha}$ since the case $\varepsilon=\frac{n}{n-\alpha}$ is trivial.

We will use the following estimate known sometimes as Kolmogorov's inequality,  if $(X,\mu)$ is a probability space (or even $\mu(X)\leq 1) $, 
and if $0<q<r$
\begin{equation}
\norm{ f }_{L^{q} ( X, \mu) } \leq \left(\frac{r}{r-q}\right)^{\frac1q} \norm{ f }_{L^{r,\infty} ( X, \mu) }. 
\label{kolmogorov}
\end{equation}

Let  $x\in\R^n$ be  such that   $M_{\alpha}\mu(x)<\infty$ and let $Q$ be a cube  containing $x$. Then
\[M_{\alpha}\mu(x) \leq M_{\alpha}(\mu\chi_{3Q} )(x)+ M_{\alpha}(\mu\chi_{(3Q)^c} )(x).
\]

Consider  first the case $\frac{n}{n-\alpha} -\varepsilon \leq 1$. 
Then, 
\[
\vint_{Q}  (M_{\alpha}\mu)(y)^{\frac{n}{n-\alpha} -\varepsilon}\,dy
\leq  \vint_{Q}  (M_{\alpha}(\mu\chi_{3Q} )(y)^{\frac{n}{n-\alpha} -\varepsilon}\,dy +\vint_{Q}  (M_{\alpha}(\mu\chi_{(3Q)^c} ))(y)^{\frac{n}{n-\alpha} -\varepsilon}\,dy =:I+II.
\]
For $I$  we use \eqref{kolmogorov} combined with \eqref{weaktypeMalpha}, 
\[
\begin{split}
I&\leq \frac{n}{n-\alpha} \frac{1}{\varepsilon} 
\|M_{\alpha}(\mu\chi_{3Q})\|_{L^{\frac{n}{n-\alpha},\infty}(Q,\frac{dx}{|Q|}) }^{\frac{n}{n-\alpha} - \varepsilon}  
\leq \frac{n}{n-\alpha}\frac{1}{\varepsilon}  \Big(5^{n-\alpha} \frac{|Q|^{\frac{\alpha}{n}}}{ |Q| } \mu(3Q) \Big)^{\frac{n}{n-\alpha} - \varepsilon}\\
&= \frac{n}{n-\alpha}\frac{1}{\varepsilon}  \Big(15^{n-\alpha} \frac{|3Q|^{\frac{\alpha}{n}}}{ |3Q| } \mu(3Q) \Big)^{\frac{n}{n-\alpha} - \varepsilon} \leq \frac{n}{n-\alpha} \frac{15^{n-\varepsilon(n-\alpha)}}{\varepsilon}  \Big(M_{\alpha}\mu(x) \Big)^{\frac{n}{n-\alpha} - \varepsilon}.
\end{split}
\]

For $II$ we use the following geometrical argument. Let $y\in Q$ and consider  $M_{\alpha}(\mu\chi_{(3Q)^c} )(y)$. 
If $R$ is a cube such that $y\in R \subset 3Q$ then 
$\frac{|R|^{\frac{\alpha}{n}}}{ |R| } \mu((3Q)^c \cap R)=0$, hence we only must take cubes $R$ containing $y$ such that 
$(3Q)^c \cap R \neq \emptyset$. Thus  these cubes $ 3R \supseteq Q$, and therefore
\[
\frac{|R|^{\frac{\alpha}{n}}}{ |R| } \mu((3Q)^c \cap R)\leq 3^{n-\alpha} \frac{|3R|^{\frac{\alpha}{n}}}{ |3R| } \mu((3Q)^c \cap 3R) \leq 
3^{n-\alpha}M_{\alpha}(\mu )(x).
\]
Hence
\[
II\leq   \Big( 3^{n-\alpha}M_{\alpha}\mu(x) \Big)^{\frac{n}{n-\alpha} - \varepsilon} 
= 3^{n-\varepsilon(n-\alpha)}  \Big(M_{\alpha}\mu(x) \Big)^{\frac{n}{n-\alpha} - \varepsilon}. 
\]
Combining, since $0<\varepsilon<\frac{n}{n-\alpha}$
\[
\vint_{Q}  (M_{\alpha}\mu)^{\frac{n}{n-\alpha} -\varepsilon}\,dx  \leq 
\frac{2n}{n-\alpha}\frac{15^{n-\varepsilon(n-\alpha)}}{\varepsilon}  \Big(M_{\alpha}\mu(x) \Big)^{\frac{n}{n-\alpha} - \varepsilon}.
\]

In the case $\frac{n}{n-\alpha} -\varepsilon > 1$, using that $(a+b)^{p}\leq 2^{p-1}(a^p+b^p)$ for any $a,b>0$, $p\ge 1$,
\[
\vint_{Q}  (M_{\alpha}\mu)(y)^{\frac{n}{n-\alpha} -\varepsilon}\,dy
\leq  2^{\frac{n}{n-\alpha} -\varepsilon-1} (I+II)=
2^{\frac{n}{n-\alpha} -\varepsilon}
\frac{n}{n-\alpha}\frac{15^{n-\varepsilon(n-\alpha)}}{\varepsilon}  \Big(M_{\alpha}\mu(x) \Big)^{\frac{n}{n-\alpha} - \varepsilon}.
\]
Finally, combining both estimates, for $0< \varepsilon<\frac{n}{n-\alpha}$, we have  
\[
\vint_{Q}  (M_{\alpha}\mu)^{\frac{n}{n-\alpha} -\varepsilon}\,dy  \leq 
2^{  \frac{n}{n-\alpha} - \varepsilon} \frac{2n}{n-\alpha}\frac{15^{n-\varepsilon(n-\alpha)}}{\varepsilon}  
\inf_{x\in Q} \Big(M_{\alpha}\mu(x) \Big)^{\frac{n}{n-\alpha} - \varepsilon}.
\]
which yields
\[
[(M_{\alpha}\mu)^{\frac{n}{n-\alpha} -\varepsilon}]_{A_1} \leq 
2^{  \frac{n}{n-\alpha} - \varepsilon} \frac{2n}{n-\alpha}\frac{15^{n-\varepsilon(n-\alpha)}}{\varepsilon}.
\]

For the proof of $(3)$, we may assume $f$ to satisfy $\|f\|_{L^{\frac{n}{\alpha},\infty}(\mathbb{R}^n)}=1$. Consider $x\in \mathbb{R}^n$ and let $Q$ be a cube containing $x$. Then, for any $L>0$,
\[
\begin{split}
\int_Q |f(y)|dy&=\int_0^\infty |\{y\in Q:|f(y)|>\lambda\}|d\lambda \\
&=\int_0^L |\{y\in Q:|f(y)|>\lambda\}|d\lambda +\int_L^\infty |\{y\in Q:|f(y)|>\lambda\}|d\lambda \\
&\leq L|Q|+\frac{1}{\frac{n}{\alpha}-1}L^{1-\frac{n}{\alpha}},
\end{split}
\]
and then 
\[
\begin{split}
|Q|^{\frac\alpha{n}-1}\int_Q|f(y)|dy&\leq |Q|^{\frac\alpha{n}-1}\left(L|Q|+\frac{1}{\frac{n}{\alpha}-1}L^{1-\frac{n}{\alpha}}\right)
\\
&= |Q|^{\frac\alpha{n} }\left(L+\frac{1}{|Q|\left(\frac{n}{\alpha}-1\right)}L^{1-\frac{n}{\alpha}}\right),
\end{split}
\]
which equals $\left(\frac{n}{\alpha}\right)'$ when choosing $L=|Q|^{-\frac{\alpha}{n}}$. This concludes the proof of $(3)$.
\end{proof}
 \subsection{The truncation method}

Another important tool that we will be using is the so called ``fractional truncation method", 
see \cite[Theorem 4.1]{DIV} and \cite[Proposition 2.14]{C}.
Similar truncation methods are originally used by Maz'ya \cite{M1}.
The proof of \cite[Theorem 3.2]{DLV} can be easily adapted
to show the following variant of this method.

\begin{theorem}\label{t.truncation}
Let  $1\le p\le q<\infty$ and let $w$, $v$ be a weights in $\R^n$. Let $Q\subset \R^n$ be a cube. 
Then the following conditions are equivalent:
\begin{itemize}
\item[(A)]
There is a constant $C_1>0$ such that inequality
\begin{align*}
\inf_{c\in\R}\sup_{t>0} t\Big(w( & \{x\in Q: |u(x)-c|>t\}) \Big)^{\frac{1}{q}}%
\le C_1
\biggl(\int_{Q}\int_{Q} \frac{\vert u(y)-u(z)\vert^p}{\lvert y-z\rvert^{n+\delta p}}
\,dz\,v(y)\,dy\biggr)^\frac{1}{p}
\end{align*}
holds for  every $u\in L^\infty(Q;wdx)$.
\item[(B)]
There is a constant $C_2>0$ such that inequality
\begin{align*}
\inf_{c\in\R} \left(\int_Q\vert u(x)-c\vert ^{q}\,w(x)\,dx\right)^{\frac{1}{q}}
\le C_2
\biggl(\int_{Q}\int_{Q} \frac{\vert u(y)-u(z)\vert^p}{\lvert y-z\rvert^{n+\delta p}} 
\,dz\,v(y)\,dy\biggr)^\frac{1}{p}
\end{align*}
holds for every measurable function $u$ in $Q$ such that
$\lvert u\rvert<\infty$ almost everywhere in $Q$.
\end{itemize}
Moreover, in 
the implication from {\rm (A)} to {\rm (B)} the constant $C_2$ is of the form $CC_1$, where $C$ does not depend on any of the parameters, 
 and in the implication from {\rm (B)} to {\rm (A)} we have
$C_1=C_2$.
\end{theorem}

\section{  Harmonic analysis and self-improving theory }\label{self-improvingmethods} 

In this section we will provide the proofs for Theorems \ref{selfIMproveGoodConstant},   \ref{selfIMproveBadConstant}, and \ref{t.hardy-MS}. We recall first some definitions and results.

\subsection{Some recent results} \label{SelfImprResults}

The plan is to derive Poincar\'e--Sobolev type inequalities from a Poincar\'e type inequality
within a general context. This will be accomplished by using some improvement results from the literature. 
Being more precise, we will be considering, as a starting point estimates, of the form 

\begin{equation}\label{eq:general-starting-a(Q)}
\vint_Q |f-f_Q|dx\leq a(Q)\,,\qquad Q\in \cQ\,,
\end{equation}
where $a:\cQ \to  [0,\infty)$ is some general functional with no restriction. 
The functional may or may not depend on $f$. In the case it does,  we will write $a_f$, but the required condition will be uniform on $f$.

The key idea is that inequality \eqref{eq:general-starting-a(Q)} enjoys a self-improving property.  It is remarkable that  this property is not so much related to the presence of a gradient on the right hand side of \eqref{eq:general-starting-a(Q)}, but to a discrete geometrical summation condition associated to the functional $a$.   The theory is very flexible and can be developed within the contexts of other geometries, see \cite{CMPR}.

We will prove  Theorems \ref{selfIMproveGoodConstant} and  \ref{selfIMproveBadConstant} using ideas from \cite{PR}  applying  a recent improvement of  \cite[Theorem 2.3]{FPW1}, namely \cite[Theorem 1.5]{CP}. 
 The following definition is  the  key   in the theory.

\begin{definition} \label{def:smallness}
 Let $0<p<\infty$, let $w$ be a weight
and let $a:\mathcal{Q}\to  [0,\infty)$. 

1)  For $s>0$, we  denote    $a\in SD_{p}^{s}(w)$, if there is a   positive   constant $c$ such that, for any family 
of disjoint dyadic subcubes $\left\lbrace Q_i\right\rbrace_i $ of any given cube $Q$,  the following inequality holds
\begin{equation}\label{eq:SDp}
\left( \sum_i a(Q_i)^p \frac{w(Q_i)}{w(Q)}  \right)^{\frac1p} \leq c 
\left(\frac{|\bigcup_iQ_i|}{|Q|} \right )^{ \frac1s }
a(Q).
\end{equation}
The best possible constant $c$, the infimum of the constants in  the last inequality, is denoted by $\|a\|_{SD_{p}^s}$.

2) We say that the functional $a$ satisfies the weighted  $D_{p}(w)$  condition  if there is a constant $c$ 
such that, for any family of disjoint dyadic cubes $\left\lbrace Q_i\right\rbrace_i $ of any given cube  $Q$, the following inequality holds
\begin{equation}\label{eq:Dp}
\left( \sum_i a(Q_i)^p \frac{w(Q_i)}{w(Q)}  \right)^{\frac1p} \leq c\,a(Q).
\end{equation}
The best possible constant $c$ above is denoted by $\|a\|_{D_{p}(w)}$. We will write in this case that $a\in D_{p}(w)$. 

\end{definition}

Next  result from \cite{CP} will be relevant.  %
 It is an improved version of the main result in \cite{FPW1} but with linear bounds in both $[w]_{A_{\infty}}$ and $p$, instead of exponential.

\begin{theorem}\label{thm:Automejoraweak}
Let $w$ be any $A_{\infty}$ weight in $\mathbb{R}^n$ and $a\in D_{p}(w)$ from \eqref{eq:Dp}. Let $f$ be a locally integrable function such that,
\begin{equation*}\label{eq:initialHypcR}
\vint_{Q} |f - f_Q|\leq a(Q) \qquad  Q\in \cQ.
\end{equation*}
Then there exists a dimensional constant $c_n>0$ such that for every $Q\in \cQ$,
\begin{equation*}
\big \| f-f_Q \big \|_{L^{p,\infty}\big( Q, \frac{w\,dx}{w(Q)}\big)} \leq  c_n\, p\, [w]_{A_{\infty}}\,\|a\|_{D_p(w)} \, a(Q).  
\end{equation*}
\end{theorem}

Using the stronger condition \eqref{eq:SDp} we can get a better result. We need to state the following very interesting improvement of one of the main results in \cite{PR}, which can be found   %
 as a consequence of the main theorem in \cite{LLO}.

\begin{theorem}\label{thm:PR-strong}

Let $w$ be any weight and let $a$ such that for some $p\ge 1$ it satisfies the weighted condition $SD^s_p(w)$ from \eqref{eq:SDp} with $s>1$ and constant $\|a\|_{\strt{1.8ex}SD_{p}^s(w)}$. 
Let $f$ be a locally integrable function such that
\begin{equation*}
\frac{1}{|Q|}\int_{Q} |f-f_{Q}| \le a(Q),
\end{equation*}
for every cube $Q$. Then, there exists a dimensional constant $c_n>0$ such that for any cube $Q$ the following inequality holds
\begin{equation}\label{eq:norm}
\|f-f_Q\|_{ L^{p}(Q,\frac{wdx}{w(Q)}) } \leq  c_n\, s \|a\|_{\strt{1.8ex}SD_{p}^s(w)}  a(Q).
\end{equation} 

\end{theorem}

\subsection{Proofs of Theorems \ref{selfIMproveGoodConstant} and  \ref{selfIMproveBadConstant}}

The starting point for both results is the same. We look for the appropriate functional.  Indeed, 
from the key initial estimate  \eqref{FPI with gain} 
\[
\vint_Q|u(x)-u_Q|\, dx\leq c_n\,  %
(1-\delta)^{\frac1p}\, [u]_{W^{\delta,p}(Q)},
\]
which holds for every $Q\in \cQ$  and any $0<\delta<1$. Now, by definition of the $A_1$ class of weights,  we have
\[
 [u]_{W^{\delta,p}(Q)} \leq    [w]^{\frac{1}{p}}_{A_1}\, \ell(Q)^\delta\,
\left( \frac{1}{w(Q)} \int_{Q}  \int_{Q}  \frac{|u(x)- u(y)|^p}{|x-y|^{n+p\delta}}\,dy\,wdx\right)^{\frac1p},
\]
and hence,  %
\begin{equation}\label{InitHyp}
\frac{1}{|Q|}\int_Q |u(x)- u_{Q}|\,dx \leq  a_u (Q),  
\end{equation}
where 
\begin{equation}\label{defFunctional_2}
 a_u (Q):= \lambda\,\ell(Q)^{\delta}   \, \left( \frac{1}{w(Q)} \int_{Q}  \int_{Q}  g(x,y)\,dy\,wdx\right)^{\frac1p}
\end{equation}
and where $\lambda$ is the constant $\lambda=c_n \, %
(1-\delta)^{\frac1p}   [w]^{\frac{1}{p}}_{A_1}\,$ 
and $\displaystyle g(x,y)=  \frac{|u(x)- u(y)|^p}{|x-y|^{n+p\delta}}$.

We will need the following properties of the functional $a$ adapting   \cite[Lemma 5.2]{PR} or more precisely Lemma 6.2 in \cite{CMPR}. To simplify the presentation we recall that for $M \geq 1$ we define $p_M^*$ by 
\begin{equation}\label{eq:pMSobolev}
\frac{1}{p} -\frac{1}{ p_M^*}=\frac{\delta}{n}\frac1M.
\end{equation}
where $M$ will be chosen. When  $M=1$, $p_1^*$  equals the fractional Sobolev exponent  $p_1^*=p_{\delta}^*$ which corresponds to the unweighted known case $w=1$ in Theorem \ref{FracSobGain}.

\begin{lemma}\label{keyCond1} Let $w\in A_1$
 and let $a$ be defined as in \eqref{defFunctional_2}.\\
1)  Let $M > 1$ and let $p_M^*$  defined as above. %
Then, $a_u\in SD_{p^*_M}^{\frac{nM'}{\delta}}(w)$, namely, 
\begin{equation}%
\left( \sum_{i}a_u(Q_i)^{ p^*_{M} }\,\frac{w(Q_i)}{w(Q)}\right)^{\frac{1}{p^*_{M}}}  \leq  [w]_{A_1}^{ \frac{ \delta}{nM} }
\,\left (   \frac{|\bigcup_iQ_i|}{|Q|}      \right )^{ \frac{\delta}{nM'}  } a_u(Q). \\
\end{equation}
and hence \,$\|a_u\|_{SD_{p^*_M}^{ \frac{nM'}{\delta} }(w)} \leq [w]_{A_{1}}^{ \frac{ \delta}{nM} }.$\\
2)  Let $M = 1$. Then $a_u \in\,D_{p_\delta^*}(w)$, namely
\begin{equation}
\left( \sum_{i}a_u(Q_i)^{ p_{\delta}^* }\,\frac{w(Q_i)}{w(Q)}\right)^{\frac{1}{p_{\delta}^*}}  \leq  [w]_{A_1}^{ \frac{ \delta}{n} }
\, a_u(Q), \\
\end{equation}
and hence $\|a_u\|_{D_{p_\delta^*}(w)} \leq [w]_{A_1}^{ \frac{ \delta }{n} }$.  
\end{lemma}

The small difference is that the functional   \eqref{defFunctional_2} is slightly different  from the usual one   
defined as   
\[
a(Q)=\ell(Q)^{\alpha}\left(\frac{1}{w(Q)}\mu(Q) \right)^{1/p}\,,
\]  
with  $\alpha, p>0$, that was considered  in 
\cite{PR}. In our case \,$\mu(Q)$\, is replaced by 
\,$\int_Q \int_Q g(x,y)dx\,dy$ with $g\geq0$. The inner integral is increasing   in $Q$, which is   enough and  the same proof can be applied here.

\begin{proof}

Recall that 
\begin{equation*}\label{defFunctional}
a_u(Q):= \lambda\,\ell(Q)^{\delta}   \, \left( \frac{1}{w(Q)} \int_{Q}  \int_{Q}  g(x,y)\,dy\,wdx\right)^{\frac1p}
\end{equation*}
and where $\lambda$ is a parameter,  $g$ is a function defined in \eqref{defFunctional_2} and  $w\in A_1$.  We will use  the following key geometric property, 
\begin{equation}\label{e.key}
\frac{|E|}{|Q|} \leq [w]_{A_1}  \frac{w(E)}{w(Q)},
\end{equation}
which is valid for any  measurable  subset   $E\subset Q$. It is   a   particular case of \eqref{two weight Ap charact} with $w=v$, $p=1$, $f=\chi_E$, although we will always restrict ourselves to subcubes rather than subsets of $Q$.

1) Case $M>1$. For simplicity in the exposition, we will omit the subindex $M$ and just use $p^*$ instead of $p^*_M$.   Then using the definition of $p^{*}$ and  the fact that $p^{*}>p$, we have  
\begin{eqnarray*}
\sum_{i}a_u(Q_i)^{p^*}w(Q_i) & = & \lambda^{p^*} \sum_{i} \left( \int_{Q_i}  \int_{Q_i}  g(x,y)\,dy\,wdx\right)^{\frac{p^*}{p}}    
\left (    \frac{\ell(Q_i)^{\delta}}{w(Q_i)^{\frac{1}{p} -\frac{1}{p^*}} } \right )^{ p^*  }          \\
& = & \lambda^{p^*} \sum_{i} \left( \int_{Q_i}  \int_{Q_i}  g(x,y)\,dy\,wdx\right)^{\frac{p^*}{p}}    
\left(    \frac{ |Q_i|^{\frac{\delta}{n}} }{w(Q_i)^{ \frac{\delta}{ nM}  } } \right )^{ p^*  } \\
& = & \lambda^{p^*} \sum_{i} \left( \int_{Q_i}  \int_{Q_i}  g(x,y)\,dy\,wdx\right)^{\frac{p^*}{p}}    \left( \frac{ |Q_i| }{w(Q_i)^{ \frac{1}{ M}  } } \right )^{ \frac{\delta p^*}{n} }  \\
&=&  \lambda^{p^*} \sum_{i} \left( \int_{Q_i}  \int_{Q_i}  g(x,y)\,dy\,wdx\right)^{\frac{p^*}{p}}    \left(  |Q_i|^{\frac{1}{ M'}}  \left(\frac{|Q_i|}{w(Q_i) } \right)^{\frac1M} \right )^{ \frac{\delta p^*}{n} }.
\end{eqnarray*}
 By applying \eqref{e.key}, we can continue to estimate 
\begin{eqnarray*}
\sum_{i}a_u(Q_i)^{p^*}w(Q_i)& \leq  & \lambda^{p^*} [w]_{A_1}^{\frac{\delta p^*}{nM} } \left( \frac{|Q|}{w(Q) }   \right)  ^{ \frac{\delta p^*}{nM}   }  \  
\sum_{i} \left( \int_{Q_i}  \int_{Q_i}  g(x,y)\,dy\,wdx\right)^{\frac{p^*}{p}} \,    |Q_i|^{ \frac{\delta p^*}{nM'} }\\
&\leq& \lambda^{p^*}
  [w]_{A_1}^{\frac{\delta  p^*}{nM} }
\left( \frac{|Q|}{w(Q) }   \right)  ^{ \frac{\delta p^*}{nM}   }  \  
\sup_i |Q_i|^{ \frac{ \delta p^* }{ n M'}   }   \sum_{i} \left( \int_{Q_i}  \int_{Q_i}  g(x,y)\,dy\,wdx\right)^{\frac{p^*}{p}}     \\
&\leq& \lambda^{p^*}
[w]_{A_1}^{\frac{\delta  p^*}{nM} }
\left( \frac{|Q|}{w(Q) }   \right)  ^{ \frac{\delta p^*}{nM}   }  \  
 |\bigcup_iQ_i|^{ \frac{ \delta p^* }{ n M'}   }   \left(\sum_{i}  \int_{Q_i}  \int_{Q_i}  g(x,y)\,dy\,wdx\right)^{\frac{p^*}{p}}     \\
&\leq& \lambda^{p^*}
[w]_{A_1}^{\frac{\delta  p^*}{nM} }
\left( \frac{|Q|}{w(Q) }   \right)  ^{ \frac{\delta p^*}{nM}   }  \  
 |\bigcup_iQ_i|^{ \frac{ \delta p^* }{ n M'}   }   \left( \int_{Q}  \int_{Q}  g(x,y)\,dy\,wdx\right)^{\frac{p^*}{p}}     \\
& = & [w]_{A_1}^{\frac{\delta  p^*}{nM} }\,   \left (   \frac{|\bigcup_iQ_i|}{|Q|}      \right )^{ \frac{ \delta p^* }{ n M'}   }\, a_u(Q)^{ p^* } w(Q).
\end{eqnarray*}
This proves part 1) of the lemma. 

The second part corresponds to the case $M=1$, which corresponds  formally to the case $M'=\infty$ before and hence without "smallness", namely without the factor\, $\displaystyle  \frac{|\bigcup_iQ_i|}{|Q|}$. 

\begin{eqnarray*}
\sum_{i}a_u(Q_i)^{p^*}w(Q_i) & = & \lambda^{p^*} \sum_{i} \left( \int_{Q_i}  \int_{Q_i}  g(x,y)\,dy\,wdx\right)^{\frac{p^*}{p}}    
\left (    \frac{\ell(Q_i)^{\delta}}{w(Q_i)^{\frac{1}{p} -\frac{1}{p^*}} } \right )^{ p^*  }          \\
& = & \lambda^{p^*} \sum_{i} \left( \int_{Q_i}  \int_{Q_i}  g(x,y)\,dy\,wdx\right)^{\frac{p^*}{p}}    
\left(    \frac{ |Q_i|^{\frac{\delta}{n}} }{w(Q_i)^{ \frac{\delta}{ n}  } } \right )^{ p^*  } \\
& = & \lambda^{p^*} \sum_{i} \left( \int_{Q_i}  \int_{Q_i}  g(x,y)\,dy\,wdx\right)^{\frac{p^*}{p}}    \left( \frac{ |Q_i| }{w(Q_i) } \right )^{ \frac{\delta p^*}{n} }  \\
&=&  \lambda^{p^*} \sum_{i} \left( \int_{Q_i}  \int_{Q_i}  g(x,y)\,dy\,wdx\right)^{\frac{p^*}{p}}    
\left(    \frac{|Q_i|}{w(Q_i) }  \right )^{ \frac{\delta p^*}{n} } \\
& \leq  & \lambda^{p^*} [w]_{A_1}^{\frac{\delta p^*}{n} } \left( \frac{|Q|}{w(Q) }   \right)  ^{ \frac{\delta p^*}{n}   }  \  
\sum_{i} \left( \int_{Q_i}  \int_{Q_i}  g(x,y)\,dy\,wdx\right)^{\frac{p^*}{p}}     \\
&\leq& \lambda^{p^*}
[w]_{A_1}^{\frac{\delta  p^*}{n} }
\left( \frac{|Q|}{w(Q) }   \right)  ^{ \frac{\delta p^*}{n}   }\,   
   \left(\sum_{i}  \int_{Q_i}  \int_{Q_i}  g(x,y)\,dy\,wdx\right)^{\frac{p^*}{p}}     \\
&\leq& \lambda^{p^*}
[w]_{A_1}^{\frac{\delta  p^*}{n} }
\left( \frac{|Q|}{w(Q) }   \right)  ^{ \frac{\delta p^*}{n}   }  \,  
  \left( \int_{Q}  \int_{Q}  g(x,y)\,dy\,wdx\right)^{\frac{p^*}{p}}     \\
& = & [w]_{A_1}^{\frac{\delta  p^*}{n} }\,  a_u(Q)^{ p^* } w(Q).
\end{eqnarray*}
\end{proof}

\begin{proof}[Proof of Theorem \ref{selfIMproveGoodConstant}]

Let $1\leq p< \tfrac{n}{\delta}$ and recall that  $p_{\delta,w}^*$ is defined by the relationship
\begin{equation}\label{eq:ptimes-A1}
\frac{1}{p} -\frac{1}{p_{\delta,w}^*}=   \frac{\delta}{n} \, \frac{1}{1+\log [w]_{A_1}}.
\end{equation} 

Fix a cube $Q$ in $\mathbb{R}^n$.  The goal is to prove the strong type $(p_{\delta,w}^*,p)$ Poincar\'e--Sobolev estimate  with $A_1$ weights:
\begin{equation}\label{FractPSBonusA1-V1}
\begin{aligned}
& \| u-u_Q\|_{L^{p_{\delta,w}^*} ( Q, \frac{wdx}{w(Q)})}\\
&\leq c_n %
\,(1-\delta)^{\frac1p} \, [w]^{1+\frac{1}{p}}_{A_1}\,\ell(Q)^{\delta}
\left(  \frac{1}{w(Q)}  \int_{Q}  \int_{Q}  \frac{|u(x)- u(y)|^p}{|x-y|^{n+p\delta}}\,dy\,wdx\right)^{\frac1p}.   
\end{aligned}
\end{equation}

We will distinguish two cases: $[w]_{A_1} > e^{\frac{1}{\delta}}$ or the contrary. 

Recall that we start with estimate \eqref{InitHyp} using the functional \eqref{defFunctional_2} which by Lemma \ref{keyCond1} part 1)  satisfies an $SD_{p^*_{M} }^{\frac{nM'}{\delta}}$-condition  from Definition \ref{def:smallness} for appropriate $M>1$  with $s=\frac{nM'}{\delta} > 1$, namely 
\begin{equation}\label{ineq:||a||}
\|a_u\|:=\|a_u\|_{SD_{p^*_M}^{ \frac{nM'}{\delta} }(w)} \leq [w]_{A_1}^{ \frac{ \delta}{nM} }.
\end{equation}
If we choose  $M= 1+ \log[w]_{A_1},$
the exponent $p^*_M$ is exactly the value $p_{\delta,w}^*$ from \eqref{eq:ptimes-A1}. 
Hence applying Theorem  \ref{thm:PR-strong}   with $p$ replaced by  $p_{\delta,w}^*$   and using estimate \eqref{ineq:||a||} we obtain  
\begin{eqnarray*}
\left( \frac{1}{ w(Q)  } \int_{ Q }   |u -u_{Q}|^{p_{\delta,w}^*}    \,wdx\right)^{\frac{1}{p_{\delta,w}^*}}  & \leq  & c_n\,s \|a_u\|\, a_u(Q) \\
& \le & c_n \frac{M'}{\delta} [w]_{A_1}^{\frac{\delta}{nM}} \, a_u(Q)\\
& \le & c_n\frac{M'}{\delta}  [w]_{A_1}^{\frac{1}{ 1+\log [w]_{A_1}  }} \, a_u(Q)\\
& \leq & c_n \frac{1+ \log[w]_{A_1}}{ \log[w]_{A_1}}\frac{1}{\delta}\, a_u(Q).
\end{eqnarray*}

Hence, if we assume first that $[w]_{A_1} > e^{\frac{1}{\delta}}$, we have 
\begin{eqnarray*}
\left( \frac{1}{ w(Q)  } \int_{ Q }   |u -u_{Q}|^{p_{\delta,w}^*}    \,wdx\right)^{\frac{1}{p_{\delta,w}^*}} 
& \leq  & 
 c_n (1+ \log[w]_{A_1}) \, a_u(Q)\\
& \leq  & c_n [w]_{A_1} \, a_u(Q).
\end{eqnarray*}
This gives \eqref{FractPSBonusA1-V1} when $[w]_{A_1} > e^{\frac{1}{\delta}}$. Assume now that $[w]_{A_1} \leq  e^{\frac{1}{\delta}}$.  Note that in this case we do not know how to prove the result we are looking for using the argument we just used directly from the strong norm. To overcome this difficulty we will use the truncation method.    That is, by Theorem~\ref{t.truncation},  it is enough to prove the weak norm version of \eqref{FractPSBonusA1-V1}, namely
\begin{equation}\label{claimFractWeak}
\norm{ u-u_Q}_{L^{p_{\delta,w}^*,\infty} ( Q, \frac{w}{w(Q)})}\leq c\,a_u(Q)\,.
\end{equation}

We observe first that for any family $\{Q_i\}_i$ of pairwise disjoint dyadic subcubes of $Q$, the following inequality follows 
from Lemma \ref{keyCond1} part 2):
\begin{equation*}%
\left(\sum_{i}a(Q_i)^{ p^*_{1} }\,\frac{w(Q_i)}{w(Q)} \right)^{\frac{1}{p^*_{1}} } \leq  e^{\frac1n}\, a_u(Q), 
\end{equation*}
since $[w]_{A_1} \leq  e^{\frac{1}{\delta}}$.  Recall from  \eqref{eq:pMSobolev} that the exponent $p_1^*$ is defined by
$\displaystyle \frac{1}{p} -\frac{1}{ p_1^*}=\frac{\delta}{n}$.

Now, applying Theorem \ref{thm:Automejoraweak}, we get
\begin{eqnarray*}
\| u-u_Q\|_{L^{p^*_{1},\infty}\big( Q, \frac{w\,dx}{w(Q)}\big)} &\leq&  c\, p^*_{1}\, [w]_{A_{\infty}}\, e^{\frac1n}\,a_u(Q).\\
&\leq & c_n\, p^*_{1}\, [w]_{A_{1}}\,a_u(Q).
\end{eqnarray*}
Consider here the same choice as before of $M=1+\log [w]_{A_1}$.  Since  $p^*_{1}>p^*_M=p_{\delta,w}^*$ (just note that $p^*_M$ is a decreasing function on $M$), by Jensen's inequality which holds 
at weak level (simply use that the inner part of what is inside   the power $\frac1p$   is less than or equal to one) we have,

\begin{eqnarray*}
\| u-u_Q\|_{L^{p_{\delta,w}^*,\infty}\big( Q, \frac{w\,dx}{w(Q)}\big)} & \leq  & \| u-u_Q\|_{L^{p^*_{1},\infty}\big( Q, \frac{w\,dx}{w(Q)}\big)}\\
&\le &
c_n\, p^*_{1} \, [w]_{A_{1}}\,a_u(Q).  
\end{eqnarray*}
This gives the claim \eqref{claimFractWeak} finishing the proof of the theorem. 
\end{proof}

We now proceed with the proof of Theorem \ref{selfIMproveBadConstant}, in which we allow ourselves to loose some sharpness in the constant regarding the $A_1$ constant of the weight involved to be able to get 
  the usual    
fractional Sobolev exponent $p_\delta^*$.

\begin{proof}[Proof of Theorem \ref{selfIMproveBadConstant}]

We have to prove 

\begin{equation*}
\begin{aligned}
& \left (\frac{1}{w(Q)} \int_Q |u-u_Q|^{p_{\delta}^*} w dx\right )^{\frac{1}{p_{\delta}^*}} \\
&\leq 
c_n\,p_{\delta}^*\, %
(1-\delta)^{\frac1p} \,
 [w]_{A_1}^{   \frac{\delta}{n} +1+\frac1p}  
\ell(Q)^{\delta} \, 
\left(\frac{1}{w(Q)}\int_{Q}  \int_{Q}  \frac{|u(x)- u(y)|^p}{|x-y|^{n+p\delta}}\,dy\,w(x)dx\right)^{\frac1p},  
\end{aligned}
\end{equation*}
where $p_\delta^*$ is the (unweighted) fractional  Sobolev exponent\,  $\frac{1}{p} -\frac{1}{ p_{\delta}^* }=\frac{\delta}{n}.$

To do this, we start as in the proof of Theorem \ref{selfIMproveGoodConstant}, namely  we begin with  \eqref{InitHyp}, namely 
\begin{equation*}
\frac{1}{|Q|}\int_Q |u(x)- u_{Q}|\,dx \leq a_u (Q),
\end{equation*}
where $ a_u $ is the functional
\begin{equation*}%
a_u(Q):= \lambda\,\ell(Q)^{\delta}   \, \left( \frac{1}{w(Q)} \int_{Q}  \int_{Q}  g(x,y)\,dy\,wdx\right)^{\frac1p},
\end{equation*}
where $\lambda=c_n  %
(1-\delta)^{\frac1p}
[w]^{\frac{1}{p}}_{A_1}\,$ and 
$\displaystyle g(x,y)=  \frac{|u(x)- u(y)|^p}{|x-y|^{n+p\delta}}$.

Also, as in the proof of Theorem  \ref{selfIMproveGoodConstant} we can use part 2) of Lemma \ref{keyCond1},  namely
if  $\{Q_i\}$ is a family of pairwise disjoint subcubes of $Q$, the following inequality holds
\begin{equation}
\left(\sum_{i} a_u(Q_i)^{ p_\delta^* }\,\frac{w(Q_i)}{w(Q)} \right)^{\frac{1}{ p_\delta^* }} \leq  [w]_{A_1}^{ \frac{ \delta}{n} }\, a_u(Q)
\end{equation}
uniformly in $u$. That is, the functional $a$ satisfies the  $D_{p_\delta^*}(w)$  condition  and further we have that $\| a_u\|_{D_{p_\delta^*}(w)} \leq [w]_{A_1}^{ \frac{ \delta }{n} }$.  Hence, by  applying Theorem \ref{thm:Automejoraweak} we have
\begin{eqnarray*}
\| u-u_Q\|_{L^{p_\delta^*,\infty} ( Q, \frac{w}{w(Q)})} &\leq& c_{n} \,p_\delta^*\, [w]_{A_1}^{ \frac{\delta}{n}+1}\, a_u(Q)\\
&=&   c_{n} \,p_\delta^*\,   %
(1-\delta)^{\frac1p}   [w]^{\frac{\delta}{n}+1+\frac1p}_{A_1}\,\,\ell(Q)^{\delta}\, \left( 
\frac{1}{w(Q)} \int_{Q}  \int_{Q}  \frac{|u(x)- u(y)|^p}{|x-y|^{n+p\delta}} \,dy\,wdx \right)^{\frac1p}.\,
\end{eqnarray*}
 An application  of the  truncation method finishes the proof, see Theorem \ref{t.truncation}.
\end{proof}

\subsection{Proof of Theorem \ref{t.hardy-MS}}

We begin this section with the last of our applications of the self-improving methods. 

\begin{proof}[Proof of Theorem \ref{t.hardy-MS}] 

Our starting point is similar as in  Theorems \ref{selfIMproveGoodConstant} and  \ref{selfIMproveBadConstant}, 
\begin{equation*}
\frac{1}{|Q|}\int_Q |u(x)- u_{Q}|\,dx \leq c_n\,%
(1-\delta)^{\frac1p}  \, \ell(Q)^{\delta}\,
\left( \frac{1}{|Q|} \int_{Q}  \int_{Q}  \frac{|u(x)- u(y)|^p}{|x-y|^{n+p\delta}}\,dy\,dx\right)^{\frac1p}.
\end{equation*}
Then, we bound the integral term in the right-hand side as follows, 
\begin{equation*}
\begin{aligned}
\ell(Q)^{\delta}\,&
\left( \frac{1}{|Q|} \int_{Q}  \int_{Q}  \frac{|u(x)- u(y)|^p}{|x-y|^{n+p\delta}}\,dy\,dx\right)^{\frac1p}\\
&=\ell(Q)^{\delta}\, \left( \frac{1}{|Q|} \int_{Q}  \int_{Q}  \frac{|u(x)- u(y)|^p}{|x-y|^{n+p\delta}}\,dy\,\frac{M_{p\delta,Q}(w)(x)}{M_{p\delta,Q}(w)(x)} dx\right)^{\frac1p} \\
&\leq \left( \frac{1}{w(Q)} \int_{Q}  \int_{Q}  \frac{|u(x)- u(y)|^p}{|x-y|^{n+p\delta}}\,dy\,M_{p\delta,Q}(w)  (x)dx\right)^{\frac1p}
\end{aligned}
\end{equation*}
by definition of $M_{p\delta,Q}$. Define now the new functional 
\begin{equation}\label{NewFunctional}
a_{u,w}(Q):= c_n  %
(1-\delta)^{\frac1p}   \, \left( \frac{1}{w(Q)} \int_{Q}  \int_{Q}  g(x,y)\,dy\,M_{p\delta,Q}(w)(x)dx\right)^{\frac1p},
\end{equation}
where $g(x,y)=  \frac{|u(x)- u(y)|^p}{|x-y|^{n+p\delta}}$.

Observe that the functional $a$ satisfies readily 
the $D_{p}(w)$ condition from  Definition \ref{def:smallness}   part 2) uniformly in $u$. Furthermore,    $\| a_{u,w}\|_{D_p(w)}\leq 1$.  Then as in the proofs of Theorems \ref{selfIMproveGoodConstant} and  \ref{selfIMproveBadConstant},  we use  Theorem \ref{thm:Automejoraweak} since we are assuming $w\in A_{\infty}$:
\begin{eqnarray*}
\| u-u_Q\|_{L^{p,\infty} ( Q, \frac{w}{w(Q)})} \leq c_n \, p\, [w]_{A_\infty}\,\| a_{u,w} \|_{D_p(w)} \, a_{u,w}(Q)
\leq c_n\,p\, [w]_{A_\infty}\, a_{u,w}(Q).
\end{eqnarray*}
Then  we can pass  to the strong norm  by using the truncation method from Theorem \ref{t.truncation}  to get
\begin{eqnarray*}
\inf_{c\in \mathbb{R}}\left( \frac{1}{ w(Q)  } \int_{ Q }   |u - c |^{p}    \,wdx\right)^{\frac{1}{p}} \leq  c_n\,   p\, [w]_{A_\infty}\, a_{u,w}(Q).
\end{eqnarray*}
This concludes the proof of Theorem \ref{t.hardy-MS}.
\end{proof} 

 We cannot  get an $L^{p_\delta^*}$ version as in the other theorems because   of    the lack of the fractional part $\ell(Q)^{\varepsilon}$ in the definition of the functional \eqref{NewFunctional} and hence  we cannot prove  a ``smallness" type  condition from Definition \ref{def:smallness} part 1).

As a consequence of Theorem \ref{t.hardy-MS} we obtain Corollary  \ref{globalversion1}, which is 
a global variant.

\begin{proof}[Proof of Corollary  \ref{globalversion1}]

By assumptions $u\in L^1_{\loc}(\R^n)\cap L^1_{\loc}(\R^n;wdx)$ and there exists an increasing sequence of cubes $\{Q_j\}_{j\in \mathbb{N}}$ such that $\mathbb{R}^n=\bigcup_{j\in\mathbb{N}} Q_j$ and
\[
\lim_{j\to\infty} \frac{1}{w(Q_j)}\int_{Q_j}u(x)w(x)dx=0.
\]
Write $c_j=\frac{1}{w(Q_j)}\int_{Q_j}u(x)w(x)dx$ for
all $j\in\N$.
 By \eqref{LqOscProp},
\[
\int_{Q_j} \lvert u(x)-c_j\rvert^p\,w(x)dx\le
2^p \inf_{c\in\R}\int_{Q_j} \lvert u(x)-c\rvert^p\,w(x)dx.
\]
 Now, by Fatou's lemma and Theorem \ref{t.hardy-MS} we get  
\[
\begin{split}
\int_{\mathbb{R}^n}|u(x)|^pw(x)dx&= \int_{\mathbb{R}^n} \lim_{j\to \infty}|u(x)- c_j|^p\chi_{Q_j}(x)w(x)dx\\
&\leq  2^p\liminf_{j\to\infty}  \inf_{c\in\R}\int_{Q_j}|u(x)-c|^pw(x)dx\\
&\leq c^p_n\,p^p \, %
(1-\delta)\,[w]_{A_\infty}^p\,\liminf_{j\to\infty} \int_{Q_j}\int_{Q_j}\frac{|u(x)-u(y)|^p}{|x-y|^{n+\delta p}}dyM_{\delta p,Q_j}( w)(x)dx\\
&\leq c_n^p \,p^p  \,%
\,(1-\delta)\, [w]_{A_\infty}^p\,\int_{\mathbb{R}^n}\int_{\mathbb{R}^n}\frac{|u(x)-u(y)|^p}{|x-y|^{n+\delta p}}dyM_{\delta p}( w)(x)dx.
\end{split}
\]
\end{proof}

\begin{proof}[Proof of Corollaries  \ref{ThmSpecialcase} and \ref{ThmSpecialcaseg}]

We recall first the second part of Lemma  \ref{PropertiesFractMax.}, for  $0\leq \alpha<n$ 
and let $0< \varepsilon\leq\frac{n}{n-\alpha}$, then $(M_{\alpha}\mu)^{\frac{n}{n-\alpha} -\varepsilon} \in A_1
$  if $M_{\alpha}\mu$ is finite almost everywhere.     The $A_1$ constant does not depend on $\mu$, more precisely,   
\begin{equation*}%
[(M_{\alpha}\mu)^{\frac{n}{n-\alpha} -\varepsilon}]_{A_1} \leq 
 2^{  \frac{n}{n-\alpha} - \varepsilon} \frac{2n}{n-\alpha}\frac{15^{n-\varepsilon(n-\alpha)}}{\varepsilon}.
\end{equation*}
Denote the weight on the left of \eqref{SpecialFractPI1} by $w(x):=(M_{n-\frac{p\delta}{\beta}}\mu (x))^\beta$,   which is finite almost everywhere.  Let $\alpha=n-\frac{p\delta}{\beta}$ and $\varepsilon:=\beta(\frac{n}{p\delta}-1)$. Since we assume that $p<\frac{n}{\delta}\min\{1,\beta\}$, $0<\alpha<n$ and $\varepsilon>0.$ Then observe that with this definition, 
\[
w(x)= (M_{\alpha}\mu(x))^{\frac{n}{n-\alpha} -\varepsilon}
\]
and hence $w\in A_1$ with constant
\[
[w]_{A_1} \leq 
 2^{  \frac{n}{n-\alpha} - \varepsilon} \frac{2n\beta}{p\delta} \frac{15^{n-\varepsilon(n-\alpha)}}{\varepsilon}
 \leq 2^{\beta} \frac{2n\beta}{p\delta}\frac{15^{n }}{\varepsilon} =  \frac{2^{\beta}}{n-p\delta}\,2n15^n.
\]

We apply now Theorem \ref{t.hardy-MS} to our weight $w$, since $A_{1}\subset A_{\infty}$ with $[w]_{A_{\infty}}\leq [w]_{A_1}$, for a cube $Q$
\begin{equation}\label{finalestim}
\begin{split}
 \inf_{c\in\R} \left(\int_{Q} \lvert u(x)-c\rvert^p\,wdx\right)^{\frac1p}
&\leq c_{n}\,p\,   \, %
(1-\delta)^{\frac1p}\, [w]_{A_1}\,
\left(\int_{Q} \int_{Q} \frac{\lvert u(x)-u(y)\rvert^p}{\lvert x-y\rvert^{n+\delta p}}\,dy\, M_{\delta p,  {Q}}(w)dx\right)^{\frac{1}{p}}\\
&\leq c_{n}\,p\,   \,%
(1-\delta)^{\frac1p}\, \frac{2^{\beta}}{n-p\delta}\,
\left(\int_{Q} \int_{Q} \frac{\lvert u(x)-u(y)\rvert^p}{\lvert x-y\rvert^{n+\delta p}}\,dy\, M_{\delta p}(w)dx\right)^{\frac{1}{p}}.
\end{split}
\end{equation}

 To bound $M_{\delta p}(w)$  we apply first part 3) and then part 1) of Lemma \ref{PropertiesFractMax.}. Indeed,  using the same notation as before  with  $\alpha:=n-\frac{p\delta}{\beta}$  and recalling that
$p\delta<n$, we get 
\begin{equation*}
\begin{aligned}
M_{\delta p}\left[(M_{\alpha}\mu)^\beta\right](x)&\leq \left(\frac{n}{\delta p}\right)'\|(M_{\alpha}\mu)^\beta\|_{L^{\frac{n}{\delta p},\infty}(\mathbb{R}^n)}\\
&=\frac{n}{n-\delta p}\|M_{\alpha}\mu \|_{L^{\frac{n\beta}{\delta p},\infty}(\mathbb{R}^n)}^\beta 
=\frac{n}{n-\delta p}\|M_{\alpha}\mu \|_{L^{\frac{n}{n-\alpha},\infty}(\mathbb{R}^n)}^\beta \\
&\leq \frac{n}{n-\delta p} 5^{(n-\alpha)\beta }\mu(\mathbb{R}^n)^\beta
=\frac{ n  5^{\delta p} }{n-\delta p}\mu(\mathbb{R}^n)^\beta\\ 
&\leq \frac{ n  5^n }{n-\delta p}\mu(\mathbb{R}^n)^\beta.  
\end{aligned}
\end{equation*}
Inserting this in \eqref{finalestim} will give Corollary \ref{ThmSpecialcase}.

Corollary \ref{ThmSpecialcaseg} follows from Corollary \ref{ThmSpecialcase} after an application of Fatou's lemma. 
\end{proof}

\subsection{Proof of Theorem \ref{A1-two weight}}
We begin as before using  \eqref{FPI with gain} as the starting point, 
\begin{equation*}
\frac{1}{|Q|}\int_Q |u(x)- u_{Q}|\,dx \leq c_n\,%
(1-\delta)^{\frac1p}  \, \ell(Q)^{\delta}\,
\left( \frac{1}{|Q|} \int_{Q}  \int_{Q}  \frac{|u(x)- u(y)|^p}{|x-y|^{n+p\delta}}\,dy\,dx\right)^{\frac1p},
\end{equation*}
which holds for every cube $Q$ and any $0<\delta<1$.

a) We bound  the integral term in the right-hand side as follows. Recall that $0<\varepsilon \leq  \delta$
\begin{equation*}
\begin{aligned}
\ell(Q)^{\delta}\,&
\left( \frac{1}{|Q|} \int_{Q}  \int_{Q}  \frac{|u(x)- u(y)|^p}{|x-y|^{n+p\delta}}\,dy\,dx\right)^{\frac1p}\\
&=\ell(Q)^{\delta}\, \left( \frac{1}{|Q|} \int_{Q}  \int_{Q}  \frac{|u(x)- u(y)|^p}{|x-y|^{n+p\delta}}\,dy\, \frac{M_{(\delta-\varepsilon)p,  {Q}}(w {\chi_Q})(x)}{M_{(\delta-\varepsilon)p,  {Q}}(w {\chi_Q})(x)}dx\right)^{\frac1p} \\
&\leq \ell(Q)^{\varepsilon}\,\left( \frac{1}{w(Q)} \int_{Q}  \int_{Q}  \frac{|u(x)- u(y)|^p}{|x-y|^{n+p\delta}}\,dy\,M_{(\delta-\varepsilon)p,  {Q}}(w {\chi_Q})(x)dx\right)^{\frac1p}
\end{aligned}
\end{equation*}
and define now the new functional 
\begin{equation}%
a_{w,u}(Q):=  \lambda_{\delta}    \,\ell(Q)^{\varepsilon}\, \left( \frac{1}{w(Q)} \int_{Q}  \int_{Q}  g(x,y)\,dy\,A(w)(x)dx\right)^{\frac1p},
\end{equation}
where $g(x,y):=  \frac{|u(x)- u(y)|^p}{|x-y|^{n+p\delta}}$, $A(w):=M_{(\delta-\varepsilon)p,  {Q}}(w {\chi_Q})$
and $\lambda_{\delta}= c_n %
(1-\delta)^{\frac1p}$. 

The key point is that this functional satisfies  the property in  Definition \ref{def:smallness} part 1),  namely 
\begin{equation}%
\left( \sum_{i}a_{w,u}(Q_i)^{ p }\,\frac{w(Q_i)}{w(Q)}\right)^{\frac{1}{p}}  \leq  
\,\left (   \frac{|\bigcup_iQ_i|}{|Q|}      \right )^{ \frac{\varepsilon}{n}  } a_{w,u}(Q) \\
\end{equation}
and hence  \,$\| a_{w,u}\|_{SD_{p}^{ \frac{n}{\varepsilon} }(w)} \leq 1.$ This property is very easy to verify. Indeed, let $\{Q_i\}$  be a disjoint family of dyadic cubes in $Q$.  Then   
\begin{eqnarray*}
\sum_{i}a_{w,u}(Q_i)^{ p }\,\frac{w(Q_i)}{w(Q)} 
& = & \sum_{i}\lambda^p_{\delta} \,\ell(Q_i)^{p\varepsilon} \frac{1}{w(Q_i)} \int_{Q_i}  \int_{Q_i}  g(x,y)\,dy\,A(w)(x)dx \,\frac{w(Q_i)}{w(Q)}
\\
& = & \,\frac{\lambda^p_{\delta}}{w(Q)}\, \sum_{i}  \,|Q_i|^{\frac{p\varepsilon}{n}} \int_{Q_i}  \int_{Q_i}  g(x,y)\,dy\,A(w)(x)dx  \\
&\leq& \, \frac{\lambda^p_{\delta}}{w(Q)}\, \left( \Big|\bigcup_iQ_i \Big|\right )^{\frac{p\varepsilon}{n}}
   \int_{Q}  \int_{Q}  g(x,y)\,dy\,A(w)(x)dx\\ 
& = & \left (   \frac{|\bigcup_iQ_i|}{|Q|}      \right )^{ \frac{p\varepsilon}{n}  }\, a_{w,u}(Q)^{ p }.
\end{eqnarray*}

Then, by Theorem \ref{thm:PR-strong} with  \,$\| a_{w,u}\|_{SD_{p}^{ \frac{n}{\varepsilon} }(w)} \leq 1$ and $s=\frac{n}{\varepsilon}$,
\begin{equation}
\|u-u_Q\|_{ L^{p}(Q,\frac{wdx}{w(Q)}) } \leq  \frac{c_n}{\varepsilon} \,  a_{w,u}(Q).
\end{equation}
This yields  \eqref{A1-two weight1}.

b) To finish the proof of the theorem we prove now  \eqref{FS-type}. Indeed, by part a) with $\frac12<\varepsilon=\delta<1$, there is a positive dimensional constant $c_n$ such that for any weight $w$, 
  \[
\begin{aligned}
\Bigg{(}\int_{Q} &\lvert u(x)-u_{Q}\rvert^p \,wdx\Bigg{)}^{1/p}\\
&\leq c_n\, \frac{ %
(1-\delta)^{\frac{1}{p}}}{ \delta } \ell(Q)^{\delta } \left(\int_{Q} \int_{Q} \frac{\lvert u(x)-u(y)\rvert^p}{\lvert x-y\rvert^{n+\delta p}}\,dy\, M(w {\chi_Q})(x)\,dx\right)^{1/p}.
\end{aligned}
\]
This  gives \eqref{FS-type} of the theorem,  since $\delta>\tfrac12$.  For \eqref{A1-two weight2} we recall that from Definition \ref{TwoWeights},  $M(w)(x)\leq [w,v]_{A_1} v(x)$ and hence we continue with, 
\[
\begin{aligned}
 \Bigg{(}\int_{Q} \lvert u(x)-u_{Q}\rvert^p \,wdx\Bigg{)}^{1/p}&\leq  c_n\,(1-\delta)^{\frac{1}{p}} \ell(Q)^{\delta }
\left(\int_{Q} \int_{Q} \frac{\lvert u(x)-u(y)\rvert^p}{\lvert x-y\rvert^{n+\delta p}}\,dy\, M(w)(x)\,dx\right)^{1/p}\\
&\leq  c_n\, (1-\delta)^{\frac{1}{p}} \ell(Q)^{\delta} [w,v]_{A_1}^{\frac1p}
\left(\int_{Q} \int_{Q} \frac{\lvert u(x)-u(y)\rvert^p}{\lvert x-y\rvert^{n+\delta p}}\,dy\, v(x)\,dx\right)^{1/p}.
\end{aligned}
\]
This concludes the proof of part b) and hence the proof of  Theorem \ref{A1-two weight}.
\qed

 \section{Weighted fractional  isoperimetric estimates and representation formula} 
In this section we will discuss the proofs of the weighted Hardy and Poincar\'e--Sobolev type inequalities stated in Section  \ref{subsec.weighted.Isop.Ineq} using this time representation formulas instead of the general self-improving type arguments in Section \ref{self-improvingmethods}.  At the end we will derive some global inequalities as a consequence of these results.

First we will present some local representation formulas which follow from general Poincar\'e type inequalities. Then we will introduce some consequences of these representation formulas.

 The following lemma is a key  in our arguments. It essentially follows from 
  \cite{FLW}, but it can also be obtained by following the proof of \cite[Lemma 4.10]{HV1}, since cubes are examples of John domains. We are interested in the tracking of the constants involved in our estimates and so   we will provide the proof here for the sake of clarity.

\begin{lemma}\label{l.riesz_poinc}
Let $Q_0$ be a cube in $\R^n$.
Assume that  $0<\alpha< n$ and  consider $0<\eta <n-\alpha$ and $1\leq r<\infty $. Let 
$u\in L^1(Q_0)$ and let $g$ be a non-negative measurable function on $Q_0$
such that for a finite constant $\kappa$,
\begin{equation}\label{kappaAssump0}
\vint_{Q}\lvert u(x)-u_{Q}\rvert\,dx
 \leq  {\kappa} \,\ell(Q)^{\alpha}   \left(\vint_Q g(x)^r\,dx\right)^{\frac1r}
\end{equation}
for every cube $Q\subset Q_0$. Then there exists a dimensional constant $c_n$
such that
\begin{equation*}
\lvert u(x)-u_{Q_0}\rvert \leq 
c_n\,  \frac{\kappa }{\alpha^{1/r'} \eta^{1/r}}
\,\ell(Q_0)^{\alpha/r'}
\left(I_\alpha(g^r\chi_{Q_0})(x)\right)^{^{\frac1r}} \\
\end{equation*}
for almost every $x\in Q_0$.  In the particular case that $\alpha=\eta<\frac{n}{2}$, 
\begin{equation*}
\lvert u(x)-u_{Q_0}\rvert \leq 
c_n\,  \frac{\kappa}{\alpha}
\,\ell(Q_0)^{\alpha/r'}
\left(I_\alpha(g^r\chi_{Q_0})(x)\right)^{^{\frac1r}} \\
\end{equation*}

\end{lemma}

\begin{proof} 
 
This result is well known but  we need to be precise with the main parameters involved. We adapt the main ideas from \cite{FLW} in the case $r=1$ and \cite{FH} when $r>1$ and we refer also to  \cite{LP}.   
Let $E\subset Q_0$ be the complement of the set of Lebesgue points of $u$ in $Q_0$. Then $\lvert E\rvert=0$.
For a fixed $x\in Q_0\setminus E$, there exists a chain $\{Q_k\}_{k\in\N}$ of nested  dyadic subcubes of $Q_0$ such that  $Q_1=Q_0$, $Q_{k+1}\subset Q_k$ with $|Q_k|=2^n |Q_{k+1}|$ for all   $k\in\N$ and
\begin{equation*}
\{x\}=\bigcap_{k\in\N} Q_k.
\end{equation*}
Then, 
\begin{equation*}
\lvert u(x)-u_{Q_0}\rvert=\left \lvert\lim_{k\to\infty} u_{Q_k} - u_{Q_1}\right\rvert\le \sum_{k\in\N} \left\lvert u_{Q_{k+1}}- u_{Q_k}\right\rvert.
\end{equation*}

 Now, using the dyadic structure of the chain and the assumption \eqref{kappaAssump0}, we obtain that 
\begin{align*}
\sum_{k\in\N} \left\lvert u_{Q_{k+1}}- u_{Q_k}\right\rvert & \le  \sum_{k\in\N} \frac{1}{|Q_{k+1}|}\int_{Q_{k+1}}\lvert u(y)-u_{Q_k}\rvert\,dy\\
& \le  2^n \sum_{k\in\N} \frac{1}{|Q_{k}|}\int_{Q_{k}} \lvert u(y)-u_{Q_k}\rvert\,dy\\
& \le  2^n\kappa \sum_{k\in\N} \ell(Q_k)^\alpha\left(\frac{1}{|Q_{k}|}\int_{Q_{k}}g(y)^r\, dy \right)^{1/r}\\ 
 &\le  2^n\kappa \left(\sum_{k\in\N} \ell(Q_k)^\alpha\right)^{1/r'} \left(\sum_{k\in\mathbb{N}} \frac{\ell(Q_k)^\alpha}{|Q_{k}|}\int_{Q_{k}}g(y)^r\, dy \right)^{1/r}\\ 
& \leq \frac{2^n}{(1-2^{-\alpha})^{1/r'}}\kappa \ell(Q_0)^{\alpha/r'}\left(\int_{Q_0}g(y)^r \sum_{k\in\N} \ell(Q_k)^{\alpha-n}\chi_{Q_k}(y)\,dy\right)^{1/r}.
\end{align*}
Note that the immediate estimate $|x-y|\le \sqrt{n}\ell(Q_k)$ produces an extra unwanted $\log$ factor when summing the series. We instead proceed as follows. 
Fix $y\in Q_0\setminus \{x\}$ and pick $0<\eta< n-\alpha$. Write $k_0(y)=\max\{j\in \mathbb{N}: 2^{j-1}\le \sqrt{n}\frac{\ell(Q_0)}{|x-y|}\}$. Then 
\begin{align*}
\sum_{k\in\N} \ell(Q_k)^{\alpha-n}\chi_{Q_k}(y) &\le  \frac{c_n}{|x-y|^{n-\alpha-\eta}}\sum_{k=1}^{k_0(y)} \ell(Q_k)^{-\eta}\chi_{Q_k}(y)\\
&\le \frac{c_n}{|x-y|^{n-\alpha-\eta}\ell(Q_0)^\eta}\sum_{k=1}^{k_0(y)} 2^{(k-1)\eta}\\
& \le 
\frac{c_n\,2^{\eta k_0(y)}}{|x-y|^{n-\alpha-\eta}\ell(Q_0)^\eta(1-2^{-\eta})}
\\&\le  \frac{c_n}{|x-y|^{n-\alpha}(1-2^{-\eta})}. %
\end{align*}
We conclude that  the desired inequality
\begin{equation*}
\begin{aligned}
|u(x)-u_{Q_0}|&\le  \frac{c_n\kappa}{(1-2^{-\alpha})^{1/r'}(1-2^{-\eta})^{1/r}}  \ell(Q_0)^{\alpha/r'} \left(\int_{Q_0} \frac{g(y)^r}{|x-y|^{n-\alpha}}\,dy\right)^{\frac1r}\\
&\leq
\frac{c_n\kappa}{\alpha^{1/r'}\eta^{1/r}}  \ell(Q_0)^{\alpha/r'} (I_\alpha(g^r\chi_{Q_0})(x))^{1/r}
\end{aligned}
\end{equation*}
holds for almost every $x\in Q_0$ \,since $\frac{1}{1-2^{-t}} <  \frac{c_n}{t}$ \quad $0<t<n$.  
 \end{proof}

As a main application of the representation formula we can derive the following  lemma.

\begin{lemma}\label{weak_local_isoperimetric}
Let $Q_0$ be a cube in $\R^n$.
Assume that $0<\alpha<n$. Let $u\in L^1(Q_0)$  and let $g$ be a non-negative measurable function on $Q_0$
such that the following representation formula holds
\[
\lvert u(x)-u_{Q_0}\rvert\le \rho\,I_\alpha(g\chi_{Q_0})(x),\qquad \text{a.e. }x\in Q_0.
\]
There exists a constant $c_n>0$ such that,  for any weight $w$
\[
\|(u-u_{Q_0})\chi_{Q_0}\|_{L^{\frac{n}{n-\alpha},\infty}(w)} \leq  \frac{c_n\,\rho}{\alpha}  \int_{Q_0} g(x)\, (Mw(x))^\frac{n-\alpha}{n}dx.
\] 
\end{lemma}

\begin{proof}
 Recall that  $\|\cdot\|_{L^{p,\infty}(w)}$ is not a norm when $p>1$ but it is well know that there is a norm $|||\cdot|||_{L^{p,\infty}(w)}$ such that
\[
|||\cdot|||_{L^{p,\infty}(w)} \leq \|\cdot\|_{L^{p,\infty}(w)} \leq p'|||\cdot|||_{L^{p,\infty}(w)}. 
\]
Hence, combining this together with Minkowski's integral inequality (this is a known result but we refer to \cite{BKS} for sharp bounds), we have 
\[
\|(u-u_{Q_0})\chi_{Q_0}\|_{L^{\frac{n}{n-\alpha},\infty}(w)}\leq  \rho  \|I_\alpha(g\chi_{Q_0})\chi_{Q_0}\|_{L^{\frac{n}{n-\alpha},\infty}(w)}.
\]
 \[
 \leq   
 {\frac{n\rho}{\alpha}}  \int_{Q_0}\left\| \frac{1}{|\cdot-y|^{n-\alpha}}\right\|_{L^{\frac{n}{n-\alpha},\infty}(w)} g(y)dy.
 \]
But
\begin{align*}
\left\| \frac{1}{|\cdot-y|^{n-\alpha}}\right\|_{L^{\frac{n}{n-\alpha},\infty}(w)}  & \leq   {c_n}  \sup_{r>0} \left ( |Q(y,r)|^{-1} w\left ( Q(y,r) \right )\right )^\frac{n-\alpha}{n}\\
& =  {c_n}  (M^cw (y))^\frac{n-\alpha}{n},
\end{align*}
where $Q(y,r)$ is the cube with midpoint $y$ and side-length $2r$.
By collecting all estimates,
\begin{equation*}
\|(u-u_{Q_0})\chi_{Q_0}\|_{L^{\frac{n}{n-\alpha},\infty}(w)} \leq  \frac{c_n\rho}{\alpha} \int_{Q_0} g(y) (Mw(y))^\frac{n-\alpha}{n}\, dy.
\end{equation*} 
\end{proof}

A combination of the two lemmata gives the proof of Theorem \ref{strong_local_isoperimetric}.

\begin{proof}[Proof of Theorem \ref{strong_local_isoperimetric}]

Making use of   \eqref{FPI with gain} in the case $p=1$, for any cube $Q$ 
\begin{align*}
\vint_Q \lvert u(x)-u_Q\rvert\,dx
&\leq c_n\, %
 (1-\delta)
\ell(Q)^\delta
\vint_Q\int_Q\frac{|u(x)-u(y)|}{|x-y|^{n+\delta}}\, dy\, dx.
\end{align*}
Then apply Lemma \ref{l.riesz_poinc} with $\alpha=\delta$, $r=1$, $\kappa=c_n\,(1-\delta)$ and
\[
g(x)=\left(\int_{Q_0} \frac{\lvert u(x)-u(y)\rvert}{\lvert x-y\rvert^{n+\delta }}\,dy\right)\chi_{Q_0}(x).
\]
Then there exists a constant $c=c(n)$ 
such that
\begin{equation*}
\lvert u(x)-u_{Q_0}\rvert \leq  c_n\,(1-\delta)\, I_\delta(g\chi_{Q_0})(x)
\end{equation*}
for almost every $x\in Q_0$. By Lemma \ref{weak_local_isoperimetric} we have 
\begin{equation*}\label{e.weak_fp}
\|(u-u_{Q_0})\chi_{Q_0}\|_{L^{\frac{n}{n-\delta},\infty}(w)} \leq   c_n\,(1-\delta)\,     \int_{Q_0}
\int_{Q_0}\frac{\lvert u(x)-u(y)\rvert}{\lvert x-y\rvert^{n+\delta}}\,dy \, (Mw(x))^\frac{n-\delta}{n}\,dx
\end{equation*}
since $\delta\ge \frac12$. 
 To finish we use the fractional truncation method, see Theorem \ref{t.truncation}. 
\end{proof}

Similarly we  prove next a global type version of Corollary \ref{globalversion2}.

\begin{proof}[Proof of Corollary  \ref{globalversion2}]

We proceed as in the proof of Corollary  \ref{globalversion1}. 
By assumption, there exists an increasing sequence of cubes $\{Q_j\}_{j\in \mathbb{N}}$ such that $\mathbb{R}^n=\bigcup_{j\in\mathbb{N}} Q_j$ such that 
\[
\lim_{j\to\infty} \frac{1}{w(Q_j)}\int_{Q_j}u(x)w(x)dx=0.
\]
Write $c_j=\frac{1}{w(Q_j)}\int_{Q_j}u(x)w(x)dx$ for
all $j\in\N$.  Using \eqref{LqOscProp}, Fatou's lemma combined with Theorem \ref{strong_local_isoperimetric}, yields 
\[
\begin{split}
\left ( \int_{\mathbb{R}^n} |u(x)|^{\frac{n}{n-\delta}}w(x)\,dx\right )^\frac{n-\delta}{n} &=\left ( \int_{\mathbb{R}^n} \liminf_{j\to\infty}\chi_{Q_j}(x)|u(x)-c_{j}|^{\frac{n}{n-\delta}}w(x)\,dx \right )^\frac{n-\delta}{n}\\
 &\le   \liminf_{j\to\infty} \left(\int_{\mathbb{R}^n}\chi_{Q_j}(x)|u(x)-c_{j}|^{\frac{n}{n-\delta}} w(x)\,dx\right )^\frac{n-\delta}{n}\\
 &=  \liminf_{j\to\infty} \left( \int_{Q_j}|u(x)-c_{j}|^{\frac{n}{n-\delta}} w(x)\,dx\right )^\frac{n-\delta}{n}\\
 &\leq c_n\,  (1-\delta) \,\liminf_{j\to\infty}  \int_{Q_j} \int_{Q_j} \frac{\lvert u(x)-u(y)\rvert}{\lvert x-y\rvert^{n+\delta }}\,dy \, (Mw(x))^\frac{n-\delta}{n}\,dx \\
 &\leq c_n\, 
(1-\delta)\,   \int_{\R^n}  \int_{\R^n} \frac{\lvert u(x)-u(y)\rvert}{\lvert x-y\rvert^{n+\delta }}\,dy \, (Mw(x))^\frac{n-\delta}{n}\,dx,
\end{split}
\]
which is the claimed inequality. 
\end{proof}

As a corollary, we obtain the following lower bound for a weighted fractional perimeter.

\begin{proof}[Proof of Corollary  \ref{fractionalCharacteristics}]
This follows by applying Corollary \ref{globalversion2} to the characteristic
function of $E$  since $L_c^{\infty}(\R^n) \subset \F_{w}$ for any weight $w$ such that $w\notin L^1(\R^n)$.
\end{proof}

\begin{proof}[Proof of Corollary  \ref{BBMwithweights}]      

Fix a cube $Q$ and  let $E$ be a measurable subset of $Q$. Then we can apply Theorem \ref{strong_local_isoperimetric} using \eqref{LqOscProp} with $u=\chi_{\strt{1.5ex}E}$ and $\frac12\leq \delta<1$, namely
\[
 \begin{split}
& \left(\frac{1}{w(Q)}\int_Q |u(x)-u_{Q,w}|^{\frac{n}{n-\delta}}w(x)\, dx\right)^{\frac{n-\delta}{n}}\\
&  \leq 
\frac{c_n(1-\delta)}{w(Q)^{\frac{n-\delta}{n}}} 
\Bigg(\int_E \int_{E}\frac{\lvert u(x)-u(y)\rvert Mw(x)^\frac{n-\delta}{n}}{\lvert x-y\rvert^{n+\delta}} dydx \\&\qquad\quad \qquad\qquad +\int_E \int_{Q\setminus E}\cdots\, dydx+\int_{Q\setminus E} \int_E\cdots\, dydx+ 
\int_{Q\setminus E} \int_{Q\setminus E} \cdots\, dydx \Bigg)\\
&  = \frac{c_n(1-\delta)}{w(Q)^{\frac{n-\delta}{n}}}\left( 
\int_E \int_{Q\setminus E} \frac{Mw(x)^\frac{n-\delta}{n}}{\lvert x-y\rvert^{n+\delta}}\,dydx +\int_{Q\setminus E} \int_{E} \frac{Mw(x)^\frac{n-\delta}{n}}{\lvert x-y\rvert^{n+\delta}}\,dydx \right).
\\
\end{split}
\]
On the other hand, 
\[
\begin{split}
&\frac{1}{w(Q)}\int_Q |u(x)-u_{Q,w}|^{\frac{n}{n-\delta}}w(x)\, dx
 \\& \qquad \geq \frac{1}{w(Q)} \int_{Q\setminus E} \left|\chi_{\strt{1.5ex}E}(x)- \frac{w(E)}{w(Q)} \right|^{\frac{n}{n-\delta}}w(x)\, dx
= \left(\frac{w(E)}{w(Q)}\right)^{\frac{n}{n-\delta}} \frac{w(Q\setminus E)}{w(Q)},
\end{split}
\]
and hence
\[
\begin{split}
&\left(\frac{w(E)}{w(Q)}\right)^{\frac{n}{n-\delta}} \frac{w(Q\setminus E)}{w(Q)} \\
& \qquad \leq 
\frac1{w(Q)} 
\left(c_n(1-\delta)
\int_E \int_{Q\setminus E} \frac{Mw(x)^\frac{n-\delta}{n}}{\lvert x-y\rvert^{n+\delta}}\,dydx 
+
c_n(1-\delta) \int_{Q\setminus E} \int_{E} \frac{Mw(x)^\frac{n-\delta}{n}}{\lvert x-y\rvert^{n+\delta}}\,dydx
\right) ^{\frac{n}{n-\delta}}.
\end{split}
\]
Repeating the same argument,  but replacing $E$ by $Q\setminus E$, we also have  
\[
\begin{split}
&\left(\frac{w(Q\setminus E)}{w(Q)}\right)^{\frac{n}{n-\delta}} \frac{w( E)}{w(Q)} \\
&\qquad \leq 
\frac1{w(Q)} 
\left(c_n(1-\delta)
\int_E \int_{Q\setminus E} \frac{Mw(x)^\frac{n-\delta}{n}}{\lvert x-y\rvert^{n+\delta}}\,dydx +
c_n(1-\delta) \int_{Q\setminus E} \int_{E} \frac{Mw(x)^\frac{n-\delta}{n}}{\lvert x-y\rvert^{n+\delta}}\,dydx
\right) ^{\frac{n}{n-\delta}}.
\end{split}
\]
Then, taking the maximum 
\[
\begin{split}
M &=  \max\left\{\left(\frac{w(Q\setminus E)}{w(Q)}\right)^{\frac{n}{n-\delta}} \frac{w( E)}{w(Q)}, 
\left(\frac{w(E)}{w(Q)}\right)^{\frac{n}{n-\delta}} \frac{w( Q\setminus E)}{w(Q)} \right\} \\
& \qquad \leq \frac1{w(Q)} 
\left(c_n(1-\delta)
\int_E \int_{Q\setminus E} \frac{Mw(x)^\frac{n-\delta}{n}}{\lvert x-y\rvert^{n+\delta}}\,dydx +
c_n(1-\delta) \int_{Q\setminus E} \int_{E} \frac{Mw(x)^\frac{n-\delta}{n}}{\lvert x-y\rvert^{n+\delta}}\,dydx
\right) ^{\frac{n}{n-\delta}}.
\end{split}
\]%
The maximum simplifies as follows 
\[
\begin{split}
M&= 
\max\left\{\frac{w(Q\setminus E)}{w(Q)}\left(\frac{w(Q\setminus E)}{w(Q)}\right)^{\frac{\delta}{n-\delta}} \frac{w( E)}{w(Q)}, 
\frac{w(E)}{w(Q)} \left(\frac{w(E)}{w(Q)}\right)^{\frac{\delta}{n-\delta}} \frac{w( Q\setminus E)}{w(Q)} \right\}\\
&=
\frac{w(Q\setminus E)}{w(Q)}\frac{w( E)}{w(Q)}\, \max\left\{ 
 \left(\frac{w(Q\setminus E)}{w(Q)}\right)^{\frac{\delta}{n-\delta}}, 
\left(\frac{w(E)}{w(Q)}\right)^{\frac{\delta}{n-\delta}}  \right\}
\\
&=
\frac{w(Q\setminus E)}{w(Q)}\frac{w( E)}{w(Q)}\, 
\max\left\{  \frac{w(Q\setminus E)}{w(Q)}, 
\frac{w(E)}{w(Q)}  \right\}^{\frac{\delta}{n-\delta}}
\end{split}
\]
and hence 
\[
\begin{split}
&\frac{w(Q\setminus E)}{w(Q)}\frac{w( E)}{w(Q)}\, 
\max\left\{  \frac{w(Q\setminus E)}{w(Q)}, 
\frac{w(E)}{w(Q)}  \right\}^{\frac{\delta}{n-\delta}}
\\
&\qquad \leq  \frac1{w(Q)} 
\left(c_n(1-\delta)
\int_E \int_{Q\setminus E} \frac{Mw(x)^\frac{n-\delta}{n}}{\lvert x-y\rvert^{n+\delta}}\,dydx +
c_n(1-\delta) \int_{Q\setminus E} \int_{E} \frac{Mw(x)^\frac{n-\delta}{n}}{\lvert x-y\rvert^{n+\delta}}\,dydx
\right) ^{\frac{n}{n-\delta}}.
\end{split}
\]
Now, since $\max\{1-\alpha,\alpha\} \geq \frac12$ with $ \alpha=\frac{w(E)}{w(Q)} \in [0,1]$,  we get
\[
\begin{split}
&\frac{w( Q\setminus E)}{w(Q)}\,\frac{w( E)}{w(Q)}\,  \frac{1}{2^{\frac{\delta}{n-\delta}}} \\
&\qquad  \leq 
 \frac1{w(Q)} 
\left(c_n(1-\delta)
\int_E \int_{Q\setminus E} \frac{Mw(x)^\frac{n-\delta}{n}}{\lvert x-y\rvert^{n+\delta}}\,dydx +
c_n(1-\delta) \int_{Q\setminus E} \int_{E} \frac{Mw(x)^\frac{n-\delta}{n}}{\lvert x-y\rvert^{n+\delta}}\,dydx
\right) ^{\frac{n}{n-\delta}}.
\end{split}
\]
and then
\[
\begin{split}
&  w( Q\setminus E)w( E) \\
&\qquad  \leq 
w(Q) 
\left(2^{\tfrac \delta n}c_n(1-\delta)
\int_E \int_{Q\setminus E} \frac{Mw(x)^\frac{n-\delta}{n}}{\lvert x-y\rvert^{n+\delta}}\,dydx +
2^{\tfrac \delta n}c_n(1-\delta) \int_{Q\setminus E} \int_{E} \frac{Mw(x)^\frac{n-\delta}{n}}{\lvert x-y\rvert^{n+\delta}}\,dydx
\right) ^{\frac{n}{n-\delta}}.
\end{split}
\]
This finishes the proof of Corollary  \ref{BBMwithweights} by  translating  $\delta$ to $\varepsilon$. 
\end{proof}

\bibliographystyle{amsalpha}

\providecommand{\bysame}{\leavevmode\hbox to3em{\hrulefill}\thinspace}
\providecommand{\MR}{\relax\ifhmode\unskip\space\fi MR }
\providecommand{\MRhref}[2]{%
  \href{http://www.ams.org/mathscinet-getitem?mr=#1}{#2}
}
\providecommand{\href}[2]{#2}

\end{document}